\documentclass[12pt]{amsart}%
\usepackage{graphicx}
\usepackage[small,nohug,heads=littlevee]{diagrams}
\usepackage{amscd}
\usepackage{geometry}
\usepackage{amsmath}
\usepackage{amsfonts}
\usepackage{amssymb}
\usepackage{hyperref}%
\providecommand{\U}[1]{\protect\rule{.1in}{.1in}}
\hypersetup{colorlinks,citecolor=blue,linkcolor=blue}
\newtheorem{theorem}{Theorem}[section]
\theoremstyle{plain}

\newtheorem{corollary}[theorem]{Corollary}

\newtheorem{lemma}[theorem]{Lemma}

\newtheorem{proposition}[theorem]{Proposition}

\theoremstyle{definition}
\newtheorem{definition}[theorem]{Definition}
\newtheorem{remark}{Remark}
\newtheorem{example}{Example}
\numberwithin{equation}{section}

\begin{document}
\title[Spaces admitting a coaxial homeomorphism]{Topological properties of spaces admitting a coaxial homeomorphism}
\author{Ross Geoghegan}
\address{Department of Mathematical Sciences, Binghamton University, Binghamton, New
York 13902-6000}
\email{ross@math.binghamton.edu}
\author{Craig Guilbault}
\address{Department of Mathematical Sciences, University of Wisconsin-Milwaukee,
Milwaukee, Wisconsin 53201}
\email{craigg@uwm.edu}
\author{Michael Mihalik}
\address{Department of Mathematics, Vanderbilt University, NashvilleTN, 37240}
\email{michael.l.mihalik@vanderbilt.edu}
\thanks{This research was supported in part by Simons Foundation Grants 207264 and
427244, CRG}
\thanks{Thanks to Phillip Guilbault for creating the illustrations in this document.}
\date{October 19, 2018}

\begin{abstract}
Wright \cite{Wr92} showed that, if a 1-ended simply connected locally compact
ANR $Y$ with pro-monomorphic fundamental group at infinity (i.e. representable
by an inverse sequence of monomorphisms) admits a $\mathbb{Z}$-action by
covering transformations, then that fundamental group at infinity can be
represented by an inverse sequence of finitely generated free groups.
Geoghegan and Guilbault \cite{GeGu12} strengthened that result, proving that
$Y$ also satisfies the crucial \emph{semistability} condition (i.e.
representable by an inverse sequence of epimorphisms).

Here we get a stronger theorem with weaker hypotheses. We drop the
\textquotedblleft pro-monomorphic hypothesis\textquotedblright\ and simply
assume that the $\mathbb{Z}$-action is generated by what we call a
\textquotedblleft coaxial\textquotedblright\ homeomorphism. In the
pro-monomorphic case every $\mathbb{Z}$-action by covering transformations is
generated by a coaxial homeomorphism, but coaxials occur in far greater
generality (often embedded in a cocompact action). When the generator is
coaxial, we obtain the sharp conclusion: $Y$ is proper 2-equivalent to the
product of a locally finite tree with $\mathbb{R}$. Even in the
pro-monomorphic case this is new: it says that, from the viewpoint of
fundamental group at infinity, the\ \textquotedblleft end\textquotedblright%
\ of $Y$ looks like the suspension of a totally disconnected compact set.

\end{abstract}
\maketitle

Let $Y$ be a simply connected, locally compact, absolute neighborhood retract
(ANR). (Recall that the class of ANR's includes such familiar and important
spaces as topological manifolds and locally finite $CW$ complexes.) Let
$\{C_{n}\}$ be an expanding sequence of compact subsets which exhausts $Y$ in
the sense that the union of the sets $C_{n}$ is the whole space. The
\textit{algebraic topology of} $Y$ \textit{at infinity} is studied by means of
the inverse sequence of spaces $\{Y-C_{n}\}$ where the bonds are inclusion
maps. So, for example, information about the $m^{\text{th}}$ homology of $Y$
at infinity would be obtained from the inverse sequence of abelian groups
$\{H_{m}(Y-C_{n})\}$. As a second example, the components at infinity are the
ends of $Y$, by which is meant (roughly) the members of the inverse sequence
of sets $\{{\pi}_{0}(Y-C_{n})\}$. All this is well-known\footnote{One source
for the general theory is \cite{Ge08}.}.

We always assume that $Y$ is simply connected. To keep things simple we assume
\textit{in this introduction only} that $Y$ has one end.

\textit{The equivariant case:} Suppose, in particular, that a group $G$ acts
\textit{cocompactly} as covering transformations on $Y$ (this implies that $G$
is finitely presented). Then, with suitable extra assumptions, the topological
invariants of $Y$ at infinity are invariants of the group $G$. The earliest
example is the number of ends of $Y$ which is a feature of $G$, independent of
the choice of $Y$; it is a classical theorem of Hopf that this number is
$0,1,2$ or $\infty$.

In this paper we add to the current understanding of the fundamental group at
infinity of $Y$, motivated particularly by the equivariant case. Pick a proper
ray $\omega:[0,\infty)\rightarrow Y$ in $Y$ as base ray, and, reparametrizing
if necessary, arrange that $\omega([n,n+1])$ lies in $Y-C_{n}$. Then we have
an inverse sequence of fundamental groups $\{{\pi}_{1}(Y-C_{n}),\omega(n)\}$,
where the bonding homomorphisms are defined using appropriate segments of
$\omega$. This\footnote{It is well-known that, up to pro-isomorphism, this is
independent of the choice of the sets $C_{n}$, though \textit{priori} it might
depend on $\omega$. However, in the semistable case (where we will find
ourselves in a moment) it is also independent of $\omega$; see \cite{Ge08}.}
is the \textit{fundamental pro-group of} $Y$ \textit{at infinity based at
}$\omega$. We are interested in finding the broadest possible hypotheses which
ensure that this is pro-isomorphic to a sequence of finitely generated free
groups with epimorphic bonding maps. The technical words describing these two
properties are \textquotedblleft semistable\textquotedblright\ (pro-isomorphic
to a sequence of epimorphisms) and \textquotedblleft
pro-free\textquotedblright\ (pro-isomorphic to a sequence of free groups).

There are many spaces $Y$ satisfying our hypotheses which lack the
semistability property\footnote{For example: cone off the infinite mapping
telescope formed by gluing together infinitely many copies of the mapping
cylinder of a degree $2$ map on the circle.}. However, none is known in the
equivariant case. In other words it is not known if a finitely presented group
exists which is not semistable at infinity.\footnote{If there were such a
group there would be certainly be a one-ended example; see \cite{Mi87}.}

By contrast, having a pro-free fundamental pro-group at infinity is a real
restriction in the equivariant case. While many groups have this property,
many do not. For example, the fundamental groups of Davis manifolds do not
have pro-free fundamental pro-groups at infinity.

Our aim is to isolate a feature of $Y$ which guarantees that the fundamental
pro-group at infinity is both semistable and pro-free. We do this by
considering an action of an infinite cyclic group $J$ on $Y$ by covering
transformations. If there is no such action then we have nothing to say, but
often there are many such actions. We denote a generator of $J$ by $j$, a
homeomorphism of $Y$. We say that such a $j$ is \textit{coaxial} if given any
compact subset $C$ of $Y$ there is a larger compact set $D$ of $Y$ such that
any loop in $Y-J\cdot D$ bounds in $Y-C$. By $J\cdot D$ we mean $\bigcup
_{m\in{\mathbb{Z}}}(j^{m}(D))$. Our main theorem is:

\begin{theorem}
\label{main} If there exists an infinite cyclic group $J$ acting as covering
transformations on $Y$ and generated by a coaxial homeomorphism then there is
a locally finite  tree $\mathbb{T}$ and a proper $2$-equivalence
$\widetilde{f}:Y\rightarrow\mathbb{T}\times\mathbb{R}$.
\end{theorem}

The point becomes clear when one notes that (1) at infinity, the product
$\mathbb{T}\times\mathbb{R}$ looks like the suspension of the (totally
disconnected) set of ends of $\mathbb{T}$, and (2) the pro-isomorphism type of
the fundamental pro-group at infinity is invariant under proper $2$%
-equivalences. Thus Theorem \ref{main} implies:

\begin{corollary}
The existence of such a coaxial $j$ ensures that $Y$ has semistable and
pro-free fundamental pro-group at infinity.
\end{corollary}

Theorem \ref{main} has a context in the literature. One says that the inverse
sequence $\{G_{n}\}$ of groups is \textit{pro-mono} if it is pro-isomorphic to
a sequence of groups whose bonds are monomorphisms. (So pro-mono is dual to
semistable.) Building on earlier work of Wright \cite{Wr92} two of us in
\cite{GeGu12} proved the following theorem:

\begin{theorem}
\label{mono} If the fundamental pro-group at infinity of $Y$ is pro-mono and
there is an infinite cyclic group $J$ acting as covering transformations on
$Y$ then $Y$ has semistable and pro-free fundamental pro-group at infinity.
\end{theorem}

Theorem \ref{mono} is a corollary of our new Theorem \ref{main} because, by a
lemma from \cite{Wr92}, when $Y$ satisfies the pro-mono hypothesis then, given
an infinite cyclic group $J$ acting as covering transformations on $Y$, the
generator $j$ of $J$ is coaxial. This indicates that the pro-mono hypothesis
is unnecessarily strong. We will see examples where some infinite cyclic
groups $J$ acting on $Y$ as covering transformations are generated by
coaxials, while others are not.

It should be noted that there is a large literature on semistability at
infinity of finitely presented groups (what we call here the equivariant
case). See, for example, \cite{Mi83}, \cite{Mi87}, \cite{MiTs92a},
\cite{MiTs92b}, \cite{Mi96a}, \cite{Mi96b}, and \cite{CoMi14}. The nature of
that literature is mostly about proving that a group $G$ formed by some group
theoretic constructions from simpler groups has the semistability property.
These theorems are by no means easy, and the widespread success tempts one to
ask if every finitely presented group is semistable at each end. We prefer to
be skeptical, and we see this paper, and our paper \cite{GGM2}, as attempts to
get to the essential topological nature of what semistability really entails.
We are of course motivated by the case where $Y$ is the universal cover of a
finite complex.

The layout of this paper is as follows. \S \ref{Sec: background} contains the
necessary background, including the algebra of inverse sequences and its use
in defining end invariants of topological spaces, such as fundamental group at
infinity. It also reviews the notions of $n$-equivalence and proper
$n$-equivalence. \S \ref{Section: Coaxial and Strongly Coaxial homeomorphisms}
discusses the new definitions that play a central role in this work:
\emph{coaxial} and \emph{strongly coaxial} homeomorphisms.
\S \ref{Section: Model spaces} describes and analyzes a collection of
\textquotedblleft model spaces\textquotedblright, like the space
$\Upsilon\times%
%TCIMACRO{\U{211d} }%
%BeginExpansion
\mathbb{R}
%EndExpansion
$ featured in Theorem \ref{main}. In \S \ref{Section: Bass-Serre theory} we
briefly describe some connections between our model spaces and Bass-Serre
theory. Our main theorems are proved by associating spaces of interest with
model spaces, whose end behavior is particularly nice. In
\S \ref{Section: Associating models to proper Z-actions}, where most of the
serious work is done, those associations are made. In
\S \ref{Section: General Conclusions}, we assemble our main conclusions in
their most general forms.

\section{Definitions and Background\label{Sec: background}}

This section contains terminology, notation, and background information to be
used throughout; it is divided into four subsections. The first reviews the
category of spaces to which this work applies; the second contains some basic
algebraic theory of inverse sequences; the third employs that theory to
describe \textquotedblleft end invariants\textquotedblright\ of noncompact
spaces; the fourth reviews the notion of proper homotopy equivalence and a
useful relaxation to \textquotedblleft proper $n$%
-equivalence\textquotedblright. Experts on these topics can safely skip ahead
to the next section; those desiring more detail should see \cite{Ge08} or
\cite{Gu13}.

\subsection{Spaces}

All spaces are assumed to be separable and metrizable. A space $Y$ is an
\emph{ANR} (\emph{absolute neighborhood retract}) if, whenever it is embedded
as a closed subset of a metric space $Z$, some neighborhood $U$ of $Y$
retracts onto $Y$. All spaces under consideration here will be locally compact
ANRs. Manifolds, locally finite CW complexes, and proper CAT(0) spaces are
special cases of locally compact ANRs.

A CW complex $Y$ is \emph{strongly locally finite} if $\left\{  C\left(
e\right)  \mid e\text{ is a cell of }Y\right\}  $ is a locally finite cover of
$Y$. Here $C\left(  e\right)  $, the \emph{carrier} of $e$ is the smallest
subcomplex containing $e$. This is a technical condition satisfied by all
finite-dimensional locally finite CW complexes and all locally finite
polyhedra. All results presented here can be obtained within these subclasses,
but for full generality, we make use of the more general condition. A complete
discussion can be found in \cite{Ge08}.

\subsection{Algebra of inverse sequences}

In this subsection arrows denote homomorphisms, with $\twoheadrightarrow$ a
surjection and $\rightarrowtail$ an injection. The symbol $\cong$ indicates an isomorphism.

Let
\[
G_{0}\overset{\lambda_{1}}{\longleftarrow}G_{1}\overset{\lambda_{2}%
}{\longleftarrow}G_{2}\overset{\lambda_{3}}{\longleftarrow}\cdots
\]
be an inverse sequence of groups. A \emph{subsequence} of $\left\{
G_{i},\lambda_{i}\right\}  $ is an inverse sequence of the form
\[
\begin{diagram}
G_{i_{0}} & \lTo^{\lambda_{i_{0}+1}\circ\cdots\circ\lambda_{i_{1}}
} & G_{i_{1}} & \lTo^{\lambda_{i_{1}+1}\circ\cdots\circ
\lambda_{i_{2}}} & G_{i_{2}} & \lTo^{\lambda_{i_{2}+1}\circ
\cdots\circ\lambda_{i_{3}}} & \cdots.
\end{diagram}
\]
In the future we denote a composition $\lambda_{i}\circ\cdots\circ\lambda_{j}$
($i\leq j$) by $\lambda_{i,j}$.

Sequences $\left\{  G_{i},\lambda_{i}\right\}  $ and $\left\{  H_{i},\mu
_{i}\right\}  $ are \emph{pro-isomorphic} if, after passing to subsequences,
there exists a commuting \textquotedblleft ladder diagram\textquotedblright:
\begin{equation}
\begin{diagram} G_{i_{0}} & & \lTo^{\lambda_{i_{0}+1,i_{1}}} & & G_{i_{1}} & & \lTo^{\lambda_{i_{1}+1,i_{2}}} & & G_{i_{2}} & & \lTo^{\lambda_{i_{2}+1,i_{3}}}& & G_{i_{3}}& \cdots\\ & \luTo & & \ldTo & & \luTo & & \ldTo & & \luTo & & \ldTo &\\ & & H_{j_{0}} & & \lTo^{\mu_{j_{0}+1,j_{1}}} & & H_{j_{1}} & & \lTo^{\mu_{j_{1}+1,j_{2}}}& & H_{j_{2}} & & \lTo^{\mu_{j_{2}+1,j_{3}}} & & \cdots \end{diagram} \label{basic ladder diagram}%
\end{equation}
Clearly an inverse sequence is pro-isomorphic to any of its subsequences. To
avoid tedious notation, we sometimes do not distinguish $\left\{
G_{i},\lambda_{i}\right\}  $ from its subsequences. Instead we assume that
$\left\{  G_{i},\lambda_{i}\right\}  $ has the properties of a preferred
subsequence---prefaced by the words \textquotedblleft after passing to a
subsequence and relabeling\textquotedblright.

The \emph{inverse limit }of $\left\{  G_{i},\lambda_{i}\right\}  $ is the
subgroup of $\prod G_{i}$ defined by
\[
\underleftarrow{\lim}\left\{  G_{i},\lambda_{i}\right\}  =\left\{  \left.
\left(  g_{0},g_{1},g_{2},\cdots\right)  \in\prod_{i=0}^{\infty}%
G_{i}\right\vert \lambda_{i}\left(  g_{i}\right)  =g_{i-1}\right\}  .
\]
Note that, for each $i$, there is a \emph{projection homomorphism}
$p_{i}:\underleftarrow{\lim}\left\{  G_{i},\lambda_{i}\right\}  \rightarrow
G_{i}$. It is a standard fact that pro-isomorphic inverse sequences have
isomorphic inverse limits, but that passing to an inverse limit can result in
a loss of information. For that reason, we prefer to work with
(pro-isomorphism classes of) inverse sequences, rather than their limits.

An inverse sequence $\left\{  G_{i},\lambda_{i}\right\}  $ is \emph{stable} if
it is pro-isomorphic to a constant inverse sequence $\left\{
H,\operatorname{id}_{H}\right\}  $, or equivalently, a sequence $\left\{
H_{i},\mu_{i}\right\}  $ where each $\mu_{i}$ is an isomorphism. In these
cases, the projection homomorphisms take $\underleftarrow{\lim}\left\{
G_{i},\lambda_{i}\right\}  $ isomorphically onto $H$ and each of the$\ H_{i}$.

If $\left\{  G_{i},\lambda_{i}\right\}  $ is pro-isomorphic to $\left\{
H_{i},\mu_{i}\right\}  $, where each $\mu_{i}$ is an epimorphism, we call
$\left\{  G_{i},\lambda_{i}\right\}  $ \emph{semi\-stable }(or
\emph{Mittag-Leffler, }or\emph{\ pro-epimorphic}). Similarly, if $\left\{
H_{i},\mu_{i}\right\}  $ can be chosen so that each $\mu_{i}$ is a
monomorphism, $\left\{  G_{i},\lambda_{i}\right\}  $ is called
\emph{pro-\allowbreak mono\-mor\-phic}. It is easy to show that an inverse
sequence that is both semi\-stable and pro-\allowbreak mono\-mor\-phic is stable.

\subsection{Ends of spaces and their algebraic invariants\label{topology}}

Proper maps and proper homotopies will be reviewed in the next subsection. In
the meantime, we will go ahead and use special cases of those concepts applied
to rays, i.e., maps $r:[0,\infty)\rightarrow X$. Those unfamiliar with the
terms can look ahead for the definitions.

A subset $N$ of a space $X$ is a \emph{neighborhood of infinity} if
$\overline{X-N}$ is compact. By a standard argument, when $X$ is an ANR and
$C\subseteq X$ is compact, $X-C$ contains at most finitely many
\emph{unbounded} components, i.e., components with noncompact closures. If
$X-C$ has both bounded and unbounded components, the situation can be
simplified by letting $C^{\prime}$ consist of $C$ together with all bounded
components. Then $C^{\prime}$ is compact, and $X-C^{\prime}$ has only
unbounded components. A neighborhood of infinity is called \emph{efficient }if
all of its components are unbounded.

Let $X=N_{0}\supseteq N_{1}\supseteq N_{2}\supseteq\cdots$ be a nested cofinal
(i.e., $\bigcap_{i=0}^{\infty}N_{i}=\varnothing$) sequence of efficient
neighborhoods of infinity in $X$. For each $i$, let $\left\{  N_{i,j}\right\}
_{j=1}^{k_{i}}$ be the set of components of $N_{i}$. Then each sequence
$\varepsilon=\left(  N_{0,j_{0}},N_{1,j_{1}},N_{2,j_{2}},\cdots\right)  $ with
the property that $N_{0,j_{0}}\supseteq N_{1,j_{1}}\supseteq N_{2,j_{2}%
}\supseteq\cdots$ determines a distinct \emph{end} of $X$. By a slight abuse
of notation, we denote the set of all such sequences by $\mathcal{E}%
\emph{nds}\left(  X\right)  $\footnote{Since our definition depends upon the
choice of $\left\{  N_{i}\right\}  $, a more precise notation might be
$\mathcal{E}\emph{nds}_{\left\{  N_{i}\right\}  }\left(  X\right)  $. A
slightly more technical, definition can be used to define $\mathcal{E}%
\emph{nds}\left(  X\right)  $ without regards to a specific cofinal sequence.
Since the two approaches are easily seen to be equivalent, we have opted for
the simpler approach.}. Clearly, $X$ is 1-ended, i.e., $\left\vert
\mathcal{E}\emph{nds}\left(  X\right)  \right\vert =1$, if and only if each
$N_{i}$ is connected. Similarly, $X$ is $k$-ended ($k<\infty$) if and only if
the number of components of $N_{i}$ stabilizes at $k$ for large $i$.

An end $\varepsilon=\left(  N_{0,j_{0}},N_{1,j_{1}},N_{2,j_{2}},\cdots\right)
$ of $X$ will be called $\pi_{1}$\emph{-null} (in $X$) if, for sufficiently
large $t$, inclusion induces the trivial homomorphism $\pi_{1}\left(
N_{2,j_{t}}\right)  \rightarrow\pi_{1}\left(  X\right)  $, i.e., loops in
$N_{2,j_{t}}$ contract in $X$.

Another method for defining ends uses proper rays.\emph{ }Declare proper
$r,r^{\prime}:[0,\infty)\rightarrow X$ to be \textquotedblleft weakly
equivalent\textquotedblright\ if $\left.  r\right\vert _{\mathbb{N}}$ is
properly homotopic to $\left.  r^{\prime}\right\vert _{\mathbb{N}}$, where
$\mathbb{N}$ denotes the natural numbers; and let $\mathcal{E}\left(
X\right)  $ denote the set of weak equivalence classes. There is a natural
bijection between $\mathcal{E}\left(  X\right)  $ and $\mathcal{E}%
\emph{nds}\left(  X\right)  $ that associates, to an equivalence class of
proper rays, the nested sequence $\varepsilon=\left(  N_{0,j_{0}},N_{1,j_{1}%
},N_{2,j_{2}},\cdots\right)  $ with the property that, for each $i$, the image
of a representative ray $r$ eventually stays in $N_{i,j_{i}}$; in that case,
we say $r$ \emph{converges to} $\varepsilon$.

Declare proper rays $r,r^{\prime}:[0,\infty)\rightarrow X$ to be
\textquotedblleft strongly equivalent\textquotedblright\ if they are properly
homotopic. The set of strong equivalence classes, $\mathcal{SE}\left(
X\right)  $, is called the set of \emph{strong ends }of $X$. This set differs
from $\mathcal{E}\left(  X\right)  $ in that rays representing the same end of
$X$ can determine distinct strong ends.

Given a proper ray $r:[0,\infty)\rightarrow X$, choose a sequence
$0=x_{0}<x_{1}<x_{2}<\cdots$, such that $r\left(  [x_{i},\infty)\right)
\subseteq N_{i}$ for all $i$ and let $p_{i}=r\left(  x_{i}\right)  $. From
there, we may construct an inverse sequence
\begin{equation}
\pi_{1}\left(  N_{0},p_{0}\right)  \overset{\lambda_{1}}{\longleftarrow}%
\pi_{1}\left(  N_{1},p_{1}\right)  \overset{\lambda_{2}}{\longleftarrow}%
\pi_{1}\left(  N_{2},p_{2}\right)  \overset{\lambda_{3}}{\longleftarrow}%
\cdots.\medskip\label{Defn: pro-pi1}%
\end{equation}
where each $\lambda_{i+1}:\pi_{1}\left(  N_{i+1},p_{i+1}\right)
\rightarrow\pi_{1}\left(  N_{i},p_{i}\right)  $ is induced by inclusion
followed by the change of base point isomorphism determined by the path
$\left.  r\right\vert _{\left[  x_{i},x_{i+1}\right]  }$. The pro-isomorphism
class of (\ref{Defn: pro-pi1}) is independent of the sequence $\left\{
N_{i}\right\}  $ and the sequence $\left\{  x_{i}\right\}  $, so we use the
notation $\operatorname{pro}$-$\pi_{1}\left(  X,r\right)  $. For multi-ended
$X$, the data in (\ref{Defn: pro-pi1}) concerns only the components
$N_{i,j_{i}}$ of the $N_{i}$ containing $p_{i}$, so (\ref{Defn: pro-pi1}) is
the same as
\begin{equation}
\pi_{1}\left(  N_{0,j_{0}},p_{0}\right)  \overset{\lambda_{1}}{\longleftarrow
}\pi_{1}\left(  N_{1,j_{1}},p_{1}\right)  \overset{\lambda_{2}}{\longleftarrow
}\pi_{1}\left(  N_{2,j_{2}},p_{2}\right)  \overset{\lambda_{3}}{\longleftarrow
}\cdots.\medskip
\end{equation}
We view sequence (\ref{Defn: pro-pi1}) as a representative of the
\emph{fundamental pro-group of the end} $\varepsilon$\emph{ with base ray }%
$r$; sometimes, for emphasis, denoting it by $\operatorname{pro}$-$\pi
_{1}\left(  \varepsilon,r\right)  $. Clearly a base ray converging to a
different end leads to entirely different information about $X$.

Even if $r$ and $r^{\prime}$ converge to the same end $\varepsilon$,
$\operatorname{pro}$-$\pi_{1}\left(  \varepsilon,r\right)  $ and
$\operatorname{pro}$-$\pi_{1}\left(  \varepsilon,r^{\prime}\right)  $ can fail
to be pro-isomorphic. It is, however, a standard fact that, if $r$ and
$r^{\prime}$ are properly homotopic, $\operatorname{pro}$-$\pi_{1}\left(
X,r\right)  $ and $\operatorname{pro}$-$\pi_{1}\left(  X,r^{\prime}\right)  $
are pro-isomorphic. More about that situation in a moment.

Given the above setup, let $\varepsilon=\left(  N_{0,j_{0}},N_{1,j_{1}%
},N_{2,j_{2}},\cdots\right)  $. We say that

\begin{enumerate}
\item \label{List item 1}$X$ is \emph{simply connected at }$\varepsilon$\emph{
}if $\operatorname{pro}$-$\pi_{1}\left(  X,r\right)  $ is pro-trivial for some
proper ray $r$ converging to $\varepsilon$,

\item \label{List item 2}$X$ is \emph{semistable }(or \emph{pro-epimorphic})
at $\varepsilon$ if $\operatorname{pro}$-$\pi_{1}\left(  X,r\right)  $ is
pro-epimorphic for some proper ray $r$ converging to $\varepsilon$, and

\item \label{List item 3}$X$ is \emph{pro-monomorphic at }$\varepsilon$ if
$\operatorname{pro}$-$\pi_{1}\left(  X,r\right)  $ is pro-monomorphic for some
proper ray $r$ converging to $\varepsilon$.\medskip
\end{enumerate}

\noindent We say that $X$ is \emph{simply connected at infinity} if $X$ is
1-ended and simply connected at that end; $X$ is \emph{semistable at infinity}
if $X$ is 1-ended and semistable at that end; and $X$ is \emph{pro-monomorphic
at infinity} if $X$ is 1-ended and pro-monomorphic at that end.

Although $\operatorname{pro}$-$\pi_{1}\left(  \varepsilon,r\right)  $ can
depend upon $r$, the question of whether an end $\varepsilon$ [1-ended space
$X$] is simply connected, semistable, or pro-monomorphic at $\varepsilon$
[infinity] is independent of the base ray converging to $\varepsilon$. For
simple connectivity and semistability, that is a consequence of the following
important fact, whose proof can be found in Chapter 16 of \cite{Ge08}.
Likewise, but for different reasons (also found in \cite{Ge08}), the
pro-monomorphic property is independent of base ray.

\begin{proposition}
\label{Prop: pro-pi semistable}Let $\varepsilon=\left(  N_{0,j_{0}}%
,N_{1,j_{1}},N_{2,j_{2}},\cdots\right)  $ determine an end of a space $X$ and
$r:[0,\infty)\rightarrow X$ be a proper ray converging to $\varepsilon$. Then
the following are equivalent:

\begin{enumerate}
\item $\operatorname{pro}$-$\pi_{1}\left(  X,r\right)  $ is semistable.

\item All proper rays in $X$ that converge to $\varepsilon$ are properly
homotopic (hence, $\operatorname{pro}$-$\pi_{1}\left(  \varepsilon,r\right)  $
is independent of $r$).\medskip
\end{enumerate}
\end{proposition}

A parallel theory of $\operatorname*{pro}$-$H_{1}\left(  X;%
%TCIMACRO{\U{2124} }%
%BeginExpansion
\mathbb{Z}
%EndExpansion
\right)  $ can be constructed in a manner similar to the above. Since base
points and connectivity are no longer issues; $\operatorname*{pro}$%
-$H_{1}\left(  X;%
%TCIMACRO{\U{2124} }%
%BeginExpansion
\mathbb{Z}
%EndExpansion
\right)  $ is represented by%
\begin{equation}
H_{1}\left(  N_{0};%
%TCIMACRO{\U{2124} }%
%BeginExpansion
\mathbb{Z}
%EndExpansion
\right)  \overset{i_{1\ast}}{\longleftarrow}H_{1}\left(  N_{1};%
%TCIMACRO{\U{2124} }%
%BeginExpansion
\mathbb{Z}
%EndExpansion
\right)  \overset{i_{2\ast}}{\longleftarrow}H_{1}\left(  N_{2};%
%TCIMACRO{\U{2124} }%
%BeginExpansion
\mathbb{Z}
%EndExpansion
\right)  \overset{i_{3\ast}}{\longleftarrow}\cdots.\medskip
\label{Defn of pro-H1}%
\end{equation}
where all maps are induced by inclusion. In the 1-ended case,
(\ref{Defn of pro-H1}) is just the abelianization of (\ref{Defn: pro-pi1}),
but in general, $\operatorname*{pro}$-$H_{1}\left(  X;%
%TCIMACRO{\U{2124} }%
%BeginExpansion
\mathbb{Z}
%EndExpansion
\right)  $ contains information about all ends of $X$. To focus on a single
end $\varepsilon=\left(  N_{0,j_{0}},N_{1,j_{1}},N_{2,j_{2}},\cdots\right)  $,
we can define $\operatorname*{pro}$-$H_{1}\left(  \varepsilon;%
%TCIMACRO{\U{2124} }%
%BeginExpansion
\mathbb{Z}
%EndExpansion
\right)  $ and represent it by the sequence
\begin{equation}
H_{1}\left(  N_{0,j_{0}};%
%TCIMACRO{\U{2124} }%
%BeginExpansion
\mathbb{Z}
%EndExpansion
\right)  \overset{i_{1\ast}}{\longleftarrow}H_{1}\left(  N_{1,j_{1}};%
%TCIMACRO{\U{2124} }%
%BeginExpansion
\mathbb{Z}
%EndExpansion
\right)  \overset{i_{2\ast}}{\longleftarrow}H_{1}\left(  N_{2,j_{2}};%
%TCIMACRO{\U{2124} }%
%BeginExpansion
\mathbb{Z}
%EndExpansion
\right)  \overset{i_{3\ast}}{\longleftarrow}\cdots.\medskip
\label{Defn of pro-H1 at epsilon}%
\end{equation}

In analogy with items (\ref{List item 1})-(\ref{List item 3}) above,
\emph{1-acyclic at }$\varepsilon$, $H_{1}$\emph{-semistable at }$\varepsilon$,
and $H_{1}$\emph{-pro-monomorphic at }$\varepsilon$ can be formulated in the
obvious ways.

\begin{remark}
Although we have focused on the $k=1$ case, the same approach leads to
definitions of $\operatorname*{pro}$-$\pi_{k}\left(  X,r\right)  $ and
$\operatorname*{pro}$-$H_{k}\left(  X;%
%TCIMACRO{\U{2124} }%
%BeginExpansion
\mathbb{Z}
%EndExpansion
\right)  $ for all $k\geq0$. The $k=0$ cases provide more ways to
\textquotedblleft count\textquotedblright\ the ends of $X$.\bigskip
\end{remark}

\subsection{Proper homotopy equivalences and proper $n$%
-equivalences\label{Subsection: Proper homotopy equivalences and n-equivalences}%
}

A map $f:X\rightarrow Y$ is \emph{proper} if $f^{-1}\left(  C\right)  $ is
compact for all compact $C\subseteq Y$. Maps $f_{0},f_{1}:X\rightarrow Y$ are
\emph{properly homotopic} is there is a proper map $H:X\times\left[
0,1\right]  \rightarrow Y$, with $H_{0}=f_{0}$ and $H_{1}=f_{1}$; in that
case, we call $H$ a \emph{proper homotopy} between $f_{0}$ and $f_{1}$ and
write $f_{0}\overset{p}{\simeq}$ $f_{1}$. Call $f:X\rightarrow Y$ is a
\emph{proper homotopy equivalence} if there exists a proper map
$g:Y\rightarrow X$ such that $gf\overset{p}{\simeq}\operatorname*{id}_{X}$ and
$fg\overset{p}{\simeq}Y$. In that case we say $X$ and $Y$ are \emph{proper
homotopy equivalent }and write $X\overset{p}{\simeq}Y$.

For our purposes, the key observation is that proper homotopy equivalences
preserve end invariants. In particular, a proper homotopy equivalence
$f:X\rightarrow Y$ induces a bijection between $\mathcal{E}\emph{nds}\left(
X\right)  $ and $\mathcal{E}\emph{nds}\left(  Y\right)  $; and if $f$ is a
proper homotopy equivalence and $r$ is a proper ray in $X$, then
$\operatorname*{pro}$-$\pi_{k}\left(  Y,f\circ r\right)  $ and
$\operatorname*{pro}$-$H_{k}\left(  Y;%
%TCIMACRO{\U{2124} }%
%BeginExpansion
\mathbb{Z}
%EndExpansion
\right)  $ are pro-isomorphic to $\operatorname*{pro}$-$\pi_{k}\left(
X,r\right)  $ and $\operatorname*{pro}$-$H_{k}\left(  X;%
%TCIMACRO{\U{2124} }%
%BeginExpansion
\mathbb{Z}
%EndExpansion
\right)  $, respectively, for all $k$.

The notion of proper homotopy equivalence often allows us to swap a generic
locally compact ANR for a locally finite polyhedron. The key tool is the
following theorem of West.

\begin{theorem}
[\cite{We77}]\label{Theorem: West's Theorem}Every compact ANR $X$ is homotopy
equivalent to a finite polyhedron; every locally compact ANR is proper
homotopy equivalent to a locally finite polyhedron.\medskip
\end{theorem}

If one is primarily interested in low-dimensional invariants, requiring a
[proper] homotopy equivalence is excessive. For $n>0$, a map between CW
complexes $f:X\rightarrow Y$ is an $n$\emph{-equivalence }if there is a map
$g:Y^{\left(  n\right)  }\rightarrow X$ such that $\left.  gf\right\vert
_{X^{\left(  n-1\right)  }}$ is homotopic to $X^{\left(  n-1\right)
}\hookrightarrow X$ and $fg$ is homotopic to $Y^{\left(  n-1\right)
}\hookrightarrow Y$. If $X$ and $Y$ are locally finite (or, more generally,
locally finite type) and each map and each homotopy is proper, we call $f$ a
\emph{proper }$n$\emph{-equivalence}. Given these conditions, $g$ is called a
\emph{[proper] }$n$\emph{-inverse} for $f$.\footnote{We have relaxed the
definitions from \cite{Ge08}, which required that (proper) n-inverses be
defined on all of $Y$. With that change, the reverse implications in
Proposition \ref{Prop: n-equivalences preserve (n-1)-invariants} become true
as well.}

\begin{example}
Every inclusion $X^{\left(  n\right)  }\hookrightarrow X$ is a proper $n$-equivalence.
\end{example}

Of key importance to this paper is the following fact which is well known to experts.

\begin{proposition}
\label{Prop: n-equivalences preserve (n-1)-invariants}A map
$f:(X,x)\rightarrow(Y,y)$ between pointed CW complexes is an $n$-equivalence
if and only if $f_{\#}:\pi_{k}\left(  X,x\right)  \rightarrow\pi_{k}\left(
Y,y\right)  $ is an isomorphism for all $k\leq n-1$. If $X$ and $Y$ are
strongly locally finite and $f$ is proper, then $f$ is a proper $n$%
-equivalence if and only if $f$ is an $n$-equivalence which induces a
pro-isomorphism between $\operatorname{pro}$-$\pi_{k}\left(  X,r\right)  $ and
$\operatorname{pro}$-$\pi_{k}\left(  Y,f\circ r\right)  $ for each proper ray
$r$ and all $k\leq n-1$.
\end{proposition}

\begin{proof}
For the absolute (non-proper) assertion, the forward implication is
straight-forward, while the reverse implication follows from a small variation
on the standard proof of the Whitehead theorem. The proper assertion follows
from the natural adaptation of those proofs to the proper category. See
\cite[Chap.16]{Ge08} for the forward implication and \cite[Props. 4.1.4 and
17.1.1]{Ge08} for the converse. Finite-dimensionality is not an issue here
since our maps need only be defined on $n$-skeleta.
\end{proof}

\begin{remark}
The converse of the proper version of Proposition
\ref{Prop: n-equivalences preserve (n-1)-invariants} can be strengthened as
follows: $f$ is a proper $n$-equivalence if $f$ is an $n$-equivalence and
there exists a representative $r$ from each element of $E\left(  X\right)  $
for which $f$ induces pro-isomorphisms between $\operatorname{pro}$-$\pi
_{k}\left(  X,r\right)  $ and $\operatorname{pro}$-$\pi_{k}\left(  Y,f\circ
r\right)  $ for all $k\leq n-1.$
\end{remark}

In a similar vein we have:

\begin{proposition}
If $f:X\rightarrow Y$ is an $n$-equivalence between CW complexes then
$f_{\ast}:H_{k}\left(  X;%
%TCIMACRO{\U{2124} }%
%BeginExpansion
\mathbb{Z}
%EndExpansion
\right)  \rightarrow H_{k}\left(  Y;%
%TCIMACRO{\U{2124} }%
%BeginExpansion
\mathbb{Z}
%EndExpansion
\right)  $ is an isomorphism for all $k\leq n-1$. If $f$ is a proper
$n$-equivalence then, in addition, $f$ induces pro-isomorphisms between
$\operatorname{pro}$-$H_{k}\left(  X;%
%TCIMACRO{\U{2124} }%
%BeginExpansion
\mathbb{Z}
%EndExpansion
\right)  $ and $\operatorname{pro}$-$H_{k}\left(  Y;%
%TCIMACRO{\U{2124} }%
%BeginExpansion
\mathbb{Z}
%EndExpansion
\right)  $ for all $k\leq n-1$.
\end{proposition}

By combining Proposition \ref{Prop: n-equivalences preserve (n-1)-invariants}
with Theorem \ref{Theorem: West's Theorem} we can extend the notion of
[proper] $n$-equivalence to locally compact ANRs: A map $f:(X,x)\rightarrow
(Y,y)$ between pointed locally compact ANRs is an $n$\emph{-equivalence} if
$f_{\#}:\pi_{k}\left(  X,x\right)  \rightarrow\pi_{k}\left(  Y,y\right)  $ is
an isomorphism for all $k\leq n-1$. If $f$ is proper, then $f$ is a
\emph{proper }$n$\emph{-equivalence} if $f$ is an $n$-equivalence which
induces pro-isomorphisms between $\operatorname{pro}$-$\pi_{k}\left(
X,r\right)  $ and $\operatorname{pro}$-$\pi_{k}\left(  Y,f\circ r\right)  $
for each proper ray $r$ and all $k\leq n-1$.

Another well known theorem (see, for example, \cite[\S 10.1]{Ge08}) plays a
useful role in this paper. Combined with Theorem \ref{Theorem: West's Theorem}
it allows us to trade ANRs for locally finite polyhedra in most of our proofs.

\begin{proposition}
\label{Prop: proper lifts to proper}Let $f:(X,x)\rightarrow(Y,y)$ be a map
between locally compact ANRs inducing an isomorphism on fundamental groups and
$\widetilde{f}:\widetilde{X}\rightarrow\widetilde{Y}$ a lift to their
universal covers. If $f$ is a [proper] homotopy equivalence then so is
$\widetilde{f}$; similarly, if $f$ is a [proper] $n$-equivalence then so is
$\widetilde{f}$.
\end{proposition}

\begin{corollary}
\label{Corollary: Equivariant West theorem}Suppose $J$ acts as covering
transformations on a locally compact ANR $X$. Then there is a $J$-equivariant
proper homotopy equivalence $g:X\rightarrow Y$ where $Y$ is a locally finite
polyhedron on which $J$ acts by simplicial covering transformations.
\end{corollary}

A fundamental application of Propositions
\ref{Prop: n-equivalences preserve (n-1)-invariants} and
\ref{Prop: proper lifts to proper} is the following.

\begin{example}
Let $X$ and $Y$ be finite 2-complexes with $\pi_{1}\left(  X\right)  \cong
G\cong\pi_{1}\left(  Y\right)  $. Then there exists a $2$-equivalence
$f:\left(  X,x\right)  \rightarrow\left(  Y,y\right)  $. Since $f$ is
(trivially) proper, $\widetilde{f}:\widetilde{X}\rightarrow\widetilde{Y}$ is a
proper $2$-equivalence; so the $0$- and $1$-dimensional end invariants, such
as the number of ends and $\operatorname{pro}$-$H_{1}$ can be attributed
directly to $G$. Modulo issues related to base rays, the same is true for
$\operatorname*{pro}$-$\pi_{1}$.
\end{example}

\section{Coaxial and strongly coaxial
homeomorphisms\label{Section: Coaxial and Strongly Coaxial homeomorphisms}}

We now take a closer look at the fundamental objects of study in this
paper---coaxial and strongly coaxial homeomorphisms.

\begin{definition}
\label{Defn: Main definition}Let $j:Y\rightarrow Y$ be a homeomorphism of a
simply connected, locally compact ANR that generates a $%
%TCIMACRO{\U{2124} }%
%BeginExpansion
\mathbb{Z}
%EndExpansion
$-action by covering transformations, and let $J=\left\langle j\right\rangle
\cong%
%TCIMACRO{\U{2124} }%
%BeginExpansion
\mathbb{Z}
%EndExpansion
$. Then

\begin{enumerate}
\item \label{Main Definition: item 1}$j$ is \emph{coaxial} if, for every
compact set $C\subseteq Y$, there is a larger compact $D\subseteq Y$ so that
loops in $Y-J\cdot D$ contract in $Y-C$, and

\item \label{Main Definition: item 2}$j$ is \emph{strongly coaxial }if, for
every compact set $C\subseteq Y$, there is a larger compact $D\subseteq Y$ so
that loops in $Y-J\cdot D$ contract in $Y-J\cdot C$.
\end{enumerate}

\noindent Under these circumstances, call $\left(  Y,j\right)  $ a
\emph{[strongly] coaxial pair}.
\end{definition}

\begin{example}
If $Y$ (as described above) is simply connected at infinity, then every such
$j$ is coaxial.
\end{example}

\begin{example}
\label{Example: Tree x R}Let $\mathbb{T}$ be a locally finite tree,
$Y=\mathbb{T}\times%
%TCIMACRO{\U{211d} }%
%BeginExpansion
\mathbb{R}
%EndExpansion
$, and $j:Y\rightarrow Y$ be translation by $1$ in the $%
%TCIMACRO{\U{211d} }%
%BeginExpansion
\mathbb{R}
%EndExpansion
$-direction. Then $j$ is strongly coaxial. Indeed, for any compact $C$, we may
choose $D\supseteq C$ to be of the form $K\times\left[  -n,n\right]  $, where
$K$ is a finite subtree. Then $J\cdot D=K\times%
%TCIMACRO{\U{211d} }%
%BeginExpansion
\mathbb{R}
%EndExpansion
$ and each component of $Y-(K\times%
%TCIMACRO{\U{211d} }%
%BeginExpansion
\mathbb{R}
%EndExpansion
)$ is of the form $N\times%
%TCIMACRO{\U{211d} }%
%BeginExpansion
\mathbb{R}
%EndExpansion
$, where $N$ is contractible. Every loop in $Y-(K\times%
%TCIMACRO{\U{211d} }%
%BeginExpansion
\mathbb{R}
%EndExpansion
)$ lies in one of these components, where it contracts missing $J\cdot C$.
\end{example}

\begin{example}
Let $Y=%
%TCIMACRO{\U{211d} }%
%BeginExpansion
\mathbb{R}
%EndExpansion
^{3}$ and $j:Y\rightarrow Y$ be translation by $1$ along the $z$-axis. Then
$j$ is coaxial but not strongly coaxial.
\end{example}

\begin{example}
The previous two examples are easily generalized. If $W$ is a simply
connected, locally compact ANR, $Y=W\times%
%TCIMACRO{\U{211d} }%
%BeginExpansion
\mathbb{R}
%EndExpansion
$, and $j:Y\rightarrow Y$ is translation by $1$ in the $%
%TCIMACRO{\U{211d} }%
%BeginExpansion
\mathbb{R}
%EndExpansion
$-direction, then $j$ is coaxial; $j$ is strongly coaxial if and only if $W$
is simply connected at each of its ends.
\end{example}

\begin{example}
\label{Example: Non-coaxial homeomorphisms}Let the free group $\mathbb{F}%
_{2}=\left\langle a,b\mid\right\rangle $ act in the usual way on its Cayley
graph $\mathbb{T}_{4}$ (the tree of constant valence 4) and let $G=\mathbb{F}%
_{2}\times%
%TCIMACRO{\U{2124} }%
%BeginExpansion
\mathbb{Z}
%EndExpansion
$ act on $\mathbb{T}_{4}\times%
%TCIMACRO{\U{211d} }%
%BeginExpansion
\mathbb{R}
%EndExpansion
$ via the diagonal action. By Example \ref{Example: Tree x R}, the generator
of $%
%TCIMACRO{\U{2124} }%
%BeginExpansion
\mathbb{Z}
%EndExpansion
$ is coaxial. But, as a homeomorphism of $\mathbb{T}_{4}\times%
%TCIMACRO{\U{211d} }%
%BeginExpansion
\mathbb{R}
%EndExpansion
$, $a$ is not. To see this, let $C=e\times\left\{  0\right\}  $ where
$e\subseteq\mathbb{T}_{4}$ is the edge connecting $1$ to $a$. Then%
\[
\left\langle a\right\rangle \cdot C=(a\text{-}\operatorname*{axis}%
)\times\{0\}\subseteq\mathbb{T}_{4}\times%
%TCIMACRO{\U{211d} }%
%BeginExpansion
\mathbb{R}
%EndExpansion
\]
and no matter how large we make the compact set $D$, there will be loops near
infinity in the plane $(b$-$\operatorname*{axis})\times%
%TCIMACRO{\U{211d} }%
%BeginExpansion
\mathbb{R}
%EndExpansion
$ lying outside $\left\langle a\right\rangle \cdot D$ which do not contract
missing $C$, since $C$ contains the origin of that plane.
\end{example}

\begin{example}
Let $BS(m,n)$ be the Baumslag-Solitar group $\left\langle a,t\mid ta^{m}%
t^{-1}=a^{n}\right\rangle $. If $K$ is the corresponding presentation
2-complex, then $\widetilde{K}\approx\mathbb{T}_{m+n}\times%
%TCIMACRO{\U{211d} }%
%BeginExpansion
\mathbb{R}
%EndExpansion
$. Viewing $BS\left(  m,n\right)  $ as the set of covering transformations of
$\widetilde{K}$ and employing arguments like those used in Examples
\ref{Example: Tree x R} and \ref{Example: Non-coaxial homeomorphisms} one sees
that $a$ is strongly coaxial, while $t$ fails to be coaxial. (A detailed
discussion of these spaces, groups, and actions can be found in \cite{GMT18}.)
\end{example}

\begin{proposition}
\label{Prop: multiples of coaxials are coaxial}Let $j:Y\rightarrow Y$ be a
non-periodic homeomorphism and $k$ a nonzero integer, then $j$ is [strongly]
coaxial if and only if $j^{k}$ is [strongly] coaxial.
\end{proposition}

\begin{proof}
If $J=\left\langle j\right\rangle $ and $J^{\prime}=\left\langle
j^{k}\right\rangle $, notice that for any $C\subseteq Y$, $J\cdot C=J^{\prime
}\cdot C^{\prime}$, where $C^{\prime}=\cup_{i=0}^{\left\vert k\right\vert
-1}j^{i}\left(  C\right)  $. From there it is easy to see that $j$ generates a
$%
%TCIMACRO{\U{2124} }%
%BeginExpansion
\mathbb{Z}
%EndExpansion
$-action by covering transformations if and only if $j^{k}$ does and that $J$
satisfies Condition \ref{Main Definition: item 1} of\ Definition
\ref{Defn: Main definition} (resp. Condition \ref{Main Definition: item 2} of
Definition \ref{Defn: Main definition}) if and only if $J^{\prime}$ does.
\end{proof}

A lemma of Wright motivated our definition of coaxial and provides a vast
collection of examples and applications.

\begin{proposition}
[\cite{Wr92}]Let $Y$ be a locally compact simply connected ANR and
$j:Y\rightarrow Y$ generate a $%
%TCIMACRO{\U{2124} }%
%BeginExpansion
\mathbb{Z}
%EndExpansion
$-action by covering transformations on $Y$. If $Y$ is pro-monomorphic at
infinity, then $j$ is coaxial.\medskip
\end{proposition}

\begin{remark}
When $Y$ is a strongly locally finite CW complex, it is clear that a cellular
homeomorphsim $j:Y\rightarrow Y$ is [strongly] coaxial if and only if its
restriction to the 2-skeleton is [strongly] coaxial.
\end{remark}

\begin{lemma}
\label{Lemma: coaxials under equivariant p.h.e.}Suppose locally compact ANRs
$X$ and $Y$ admit $%
%TCIMACRO{\U{2124} }%
%BeginExpansion
\mathbb{Z}
%EndExpansion
$-actions generated by homeomorphisms $j$ and $j^{\prime}$, respectively and
$f:X\rightarrow Y$ is an equivariant ($fj=j^{\prime}f$) proper homotopy
equivalence. If $\left(  X,j\right)  $ is a [strongly] coaxial pair, then so
is $\left(  Y,j^{\prime}\right)  $. In fact, the same conclusions hold if we
assume only that $f$ is a proper 2-equivalence.
\end{lemma}

\begin{proof}
First we will prove the initial assertion for coaxial homeomorphisms; the
analog for strongly coaxial $f$ is similar. Afterwards we generalize to proper
$2$-equivalences.

Let $H:Y\times\left[  0,1\right]  \rightarrow Y$ be an equivariant proper
homotopy with $H_{0}=\operatorname*{id}_{Y}$ and $H_{1}=fg$, where
$g:Y\rightarrow X$ is an equivariant proper homotopy inverse for $f$. For an
arbitrary compact $C\subseteq Y$, choose compact $C^{\prime}\supseteq C$ such
that $H\left(  (Y-C^{\prime})\times\left[  0,1\right]  \right)  \subseteq
Y-C$. By hypothesis, there exists a compact $Z\supseteq f^{-1}\left(
C^{\prime}\right)  $ such that loops in $X-J\cdot Z$ contract missing
$f^{-1}\left(  C^{\prime}\right)  $. Let $D=g^{-1}\left(  Z\right)  $ and note
that $J^{\prime}\cdot D=g^{-1}(J\cdot Z)$. If $\alpha$ is a loop in
$Y-J^{\prime}\cdot D$ then $g(\alpha)$ is a loop in $X-J\cdot Z$, so there is
a singular disk $\delta\subseteq X-f^{-1}\left(  C^{\prime}\right)  $ bounding
$g(\alpha)$; hence, $f\left(  \delta\right)  $ is a singular disk in
$Y-C^{\prime}$ bounding $fg(\alpha)$. By the choice of $C^{\prime}$, there is
a singular annulus in $Y-C$ cobounded by $\alpha$ and $fg(\alpha)$. The union
of that annulus with $f(\delta)$ contracts $\alpha$ in $Y-C$.

When $f$ is only assumed to be proper 2-equivalence, use the initial assertion
together with Corollary \ref{Corollary: Equivariant West theorem} to switch to
the case where $X$ and $Y$ are locally finite polyhedra. In that setting the
line of reasoning used above can be applied within the 2-skeleta of $X$ and
$Y$ to obtain the desired conclusions.
\end{proof}

The next proposition tells us that, in the appropriate context, being
[strongly] coaxial is a group-theoretic property.

\begin{proposition}
\label{Prop: coaxial is a group theoretic property}If a group $G$ acts
cocompactly as covering transformations on simply connected ANRs $X$ and $Y$,
and there exists $g\in G$ such that $\left(  X,g\right)  $ is a [strongly]
coaxial pair, then $\left(  Y,g\right)  $ is a [strongly] coaxial pair.
\end{proposition}

\begin{proof}
Since $G\backslash X$ is a compact ANR, by \ref{Theorem: West's Theorem},
there exists a homotopy equivalence $\lambda:G\backslash X\rightarrow K$,
where $K$ is a finite CW complex; moreover, $\lambda$ lifts to a
$G$-equivariant proper homotopy equivalence $\widetilde{\lambda}%
:X\rightarrow\widetilde{K}$. Similarly, there exists a homotopy equivalence
$\mu:G\backslash Y\rightarrow L$, where $L$ is a finite CW complex; and a
corresponding $G$-equivariant proper homotopy equivalence $\widetilde{\mu
}:Y\rightarrow\widetilde{L}$.

If $\left(  X,g\right)  $ is a [strongly] coaxial pair, then by Lemma
\ref{Lemma: coaxials under equivariant p.h.e.}, so is $\left(  \widetilde{K}%
,g\right)  $ and, hence, $\left(  \widetilde{K}^{(2)},g\right)  $. Choose a
map $f:K^{(2)}\rightarrow L^{(2)}$ inducing the identity isomorphism on
fundamental groups. Then $f$ is a 2-equivalence, which lifts to a proper
$G$-equivariant $2$-equivalence $\widetilde{f}:\widetilde{K}^{(2)}%
\rightarrow\widetilde{L}^{(2)}$. Apply Lemma
\ref{Lemma: coaxials under equivariant p.h.e.} again to complete the proof.
\end{proof}

In light of Proposition \ref{Prop: coaxial is a group theoretic property},
define an element $g$ of a finitely presentable group $G$ to be
\emph{[strongly] coaxial} if for some (hence, for all) cocompact $G$-action by
covering transformations on a simply connected ANR $Y$, $g$ is a [strongly]
coaxial homeomorphism.

\begin{proposition}
\label{Prop: central elements are coaxial}For finitely presentable $G$, every
non-torsion element of the center $Z\left(  G\right)  \trianglelefteq G$ is coaxial.
\end{proposition}

\begin{proof}
Let $K$ be a finite presentation $2$-complex and $G$ act on its universal
cover $X$ in the usual way. View elements of $G$ as covering transformations
and note that the action is proper and cocompact. Let $j\in Z\left(  G\right)
$, $J=\left\langle j\right\rangle \vartriangleleft G$, and $C\subseteq X$ a
finite subcomplex with $G\cdot C=X$.

For each $g\in G$ and $F\subseteq X$, let $\overline{g}$ denote the coset
$gJ=Jg$ and note that:

\begin{enumerate}
\item[i)] $\overline{g}\cdot F=g(J\cdot F)=J\cdot(gF)$,

\item[ii)] $%
%TCIMACRO{\tbigcup \nolimits_{_{\overline{g}\in G/J}}}%
%BeginExpansion
{\textstyle\bigcup\nolimits_{_{\overline{g}\in G/J}}}
%EndExpansion
\overline{g}C=X$, and

\item[iii)] $J\cdot C\cap\overline{g}\cdot F\neq\varnothing$ if and only if
$\exists$ $h\in\overline{g}$ such that $C\cap hF\neq\varnothing$.
\end{enumerate}

By a small variation on the argument presented in \cite[Cor.1.5]%
{Gu14}\footnote{It is well-known that when $K$ is a finite $K\left(
G,1\right)  $ complex and $\omega$ is a non-trivial element of the center of
$G$, then the corresponding covering transformation $f_{\omega}$ on the
universal cover is properly homotopic to the identity map. We are using here
the fact that, even if $K$ is merely a finite connected 2-complex, the same
holds on the 1-skeleton.}, the inclusion $X^{\left(  1\right)  }%
\hookrightarrow X$ is $G$-equivariantly homotopic to $\left.  j\right\vert
_{X^{\left(  1\right)  }}:X^{\left(  1\right)  }\rightarrow X$. Let $\Gamma$
be such a homotopy. Since $C$ is a finite subcomplex, there a finite
subcomplex $E\subseteq X$ for which $\Gamma\left(  C^{\left(  1\right)
}\times\left[  0,1\right]  \right)  \subseteq E$, ie, $E$ contains every track
of the homotopy that begins in $C^{\left(  1\right)  }$. By $G$-equivariance,

\begin{enumerate}
\item[iv)] $J\cdot E$ contains $\Gamma\left(  J\cdot C^{\left(  1\right)
}\times\left[  0,1\right]  \right)  $.
\end{enumerate}

By properness of the action, there is a finite set $A\subseteq G$ so that
$C\cap\alpha E\neq\varnothing$ if and only if $\alpha\in A$. So by observation
(iii), the only cosets $\overline{g}$ for which $J\cdot C\cap\overline{g}\cdot
E\neq\varnothing$ are those with a representative in $A$. Let $D=\cup
_{\alpha\in A}\alpha C$ and note that $J\cdot D=\cup_{\alpha\in A}%
\overline{\alpha}C$. Then, by item (iii) and the definitions of $D$ and $E$,

\begin{enumerate}
\item[v)] If $\overline{g}C$ contains a point of $X-J\cdot D$, then
$\overline{g}E\subseteq X-J\cdot C$.\medskip
\end{enumerate}

\noindent\textbf{Claim.} $D$\emph{ satisfies the definition of coaxial for the
compactum }$C$\emph{.}\medskip

Let $\alpha$ be a loop in $X-J\cdot D$. By a small push, we may assume
$\alpha\subseteq X^{\left(  1\right)  }$. (In fact, we should choose $\alpha$
outside a slightly larger $J\cdot D^{\prime}$ so that, after this push,
$\alpha$ misses the current $J\cdot D$.) Since $X$ is simply connected,
$\alpha$ bounds a singular disk $\Delta$ in $X$; and by properness, there
exists $k>0$ so that $j^{k}\Delta$ misses $C$. Hence $j^{k}\alpha$ contracts
missing $C$. We will complete the proof by homotoping $\alpha$ to $j^{k}%
\alpha$ in $X-J\cdot C$. That homotopy, followed by the contraction of
$j^{k}\alpha$ along $j^{k}\Delta$ then gives the desired contraction of
$\alpha$ in $X-C$.

By item (ii), each point $x$ on the curve $\alpha$ lies in some $\overline
{g_{x}}C$ and since $x\in X-J\cdot D$, item (v) assures that $\overline{g_{x}%
}E\subseteq X-J\cdot C$. By item (iv) and $G$-equivariance of $\Gamma$, the
entire track of $x$ under $\Gamma$ lies in $\overline{g_{x}}E$ and hence
misses $J\cdot C$. So applying $\Gamma$ to $\alpha$ produces a homotopy of
$\alpha$ to $j\alpha$ in $X-J\cdot C$. By equivariance, the loop $j\alpha$
again lies in $X-J\cdot D$, so we may repeat this procedure. Continuing
inductively and concatenating, we obtain a homotopy in $X-J\cdot C$ from
$\alpha$ to $j^{k}\alpha$.
\end{proof}

\begin{corollary}
If $G$ is finitely presentable and $C\trianglelefteq G$ is infinite cyclic,
then each nontrivial element of $C$ is coaxial.
\end{corollary}

\begin{proof}
Let $S$ be the subgroup of $G$ generated by the set of all squares. Since the
conjugate of a square is a square, $S$ is normal. Furthermore, since every
element of $G/S$ has order 2, if $\overline{g}$ and $\overline{h}$ are cosets
of $S$, then
\[
\overline{g}\overline{h}\overline{g}^{-1}\overline{h}^{-1}=\overline
{g}\overline{h}\overline{g}\overline{h}=\overline{gh}\overline{gh}=1.
\]
It follows that $G/S$ is a finite abelian group; so $S$ has finite index in
$G$.

Let $C=\left\langle t\right\rangle \trianglelefteq G$ as in the hypothesis.
Then each element of $G$ conjugates $t$ to $t$ or $t^{-1}$; so each element of
$S$ conjugates $t$ to itself. It follows that $t^{2}\in Z\left(  S\right)  $,
and since $S$ acts cocompactly as covering transformations on the same space
as $G$, Proposition \ref{Prop: central elements are coaxial} implies that
$t^{2}$ is coaxial. By Proposition
\ref{Prop: multiples of coaxials are coaxial}, all elements of $C$ are
coaxial.\medskip
\end{proof}

We close this section by returning to the standard situation where
$J=\left\langle j\right\rangle $ is infinite cyclic and $p:Y\rightarrow
J\backslash Y$ is the corresponding covering projection. For $A\subseteq
J\backslash Y$, let $\widetilde{A}=p^{-1}\left(  A\right)  $. The following
easy observations will be useful as we proceed.

\begin{lemma}
\label{Lemma: Basic Lemma 1}Given the setup of Definition
\ref{Defn: Main definition},

\begin{enumerate}
\item if $C\subseteq Y$ and $A=p\left(  C\right)  $, then $\widetilde{A}%
=J\cdot C$, and

\item if $A\subseteq J\backslash Y$ is compact, then there is a compact
$C\subseteq Y$ such that $\widetilde{A}=J\cdot C$.

\item $j$ is strongly coaxial if and only if, for every compact $A\subseteq
J\backslash Y$, there is a larger compact $B\subseteq J\backslash Y$ such that
loops in $(J\backslash Y)-B$ that lift to loops in $Y$ contract in
$(J\backslash Y)-A$.
\end{enumerate}
\end{lemma}

\section{Model spaces\label{Section: Model spaces}}

The main theorems of this paper will be proved by comparing spaces of
interest---simply connected, locally compact ANRs admitting $%
%TCIMACRO{\U{2124} }%
%BeginExpansion
\mathbb{Z}
%EndExpansion
$-actions by covering transformations ---to custom-made representatives of a
class of easily understood \textquotedblleft model spaces\textquotedblright.
In this section, we construct and analyze the model spaces.

Each model evolves in three stages. First there is a \textquotedblleft model
tree\textquotedblright, which is rooted and locally finite with no leaves, and
comes equipped with a labeling of the edges by nonnegative integers (subject
to certain rules). Each model tree contains instructions for the second stage,
a \textquotedblleft model base space\textquotedblright\ which has infinite
cyclic fundamental group. The third stage, the \textquotedblleft model $%
%TCIMACRO{\U{2124} }%
%BeginExpansion
\mathbb{Z}
%EndExpansion
$-space\textquotedblright, is the universal cover of the second stage. We now
provide details.

\subsection{Model Trees}

A \emph{model tree} is a pair $\left(  \Gamma,\mathcal{K}\right)  $ where
$\Gamma$ is locally finite leafless tree with root vertex $v_{0,1}$ and
$\mathcal{K}:\operatorname*{Edges}\left(  \Gamma\right)  \rightarrow\left\{
0,1,2,\cdots\right\}  $ is a labeling function satisfying:

\begin{itemize}
\item[(i)] If a reduced edge path in $\Gamma$, beginning at $v_{0,1}$,
contains an edge with label $0$, then each subsequent edge also has label $0$.
\end{itemize}

Edges labeled $0$ are called \emph{null edges. }Condition (i) ensures that the
subgraph $\Gamma^{+}$ consisting of $v_{0,1}$ and all non-null edges and their
vertices is a rooted subtree $\Gamma^{+}\leq\Gamma$; call it the
\emph{positive subtree}. In our diagrams, edges of $\Gamma^{+}$ are indicated
with solid lines and null edges with dashed lines. See Figure
\ref{Figure: Model tree}.%
%TCIMACRO{\FRAME{ftbpFU}{3in}{2.0003in}{0pt}{\Qcb{Example of a model tree
%$\Gamma$.}}{\Qlb{Figure: Model tree}}{coaxial-fig1.eps}%
%maintain-aspect-ratio TRUE;  display "USEDEF";  valid_file "F";  width 3in;
%height 2.0003in;  depth 0pt;  original-width 0pt;  original-height 0pt;
%cropleft "0";  croptop "1";  cropright "1";  cropbottom "0";
%filename 'Coaxial-fig1.eps';file-properties "XNPEU";}} }%
%BeginExpansion
\begin{figure}[ptb]%
\centering
\includegraphics[
height=2.0003in,
width=3in
]%
{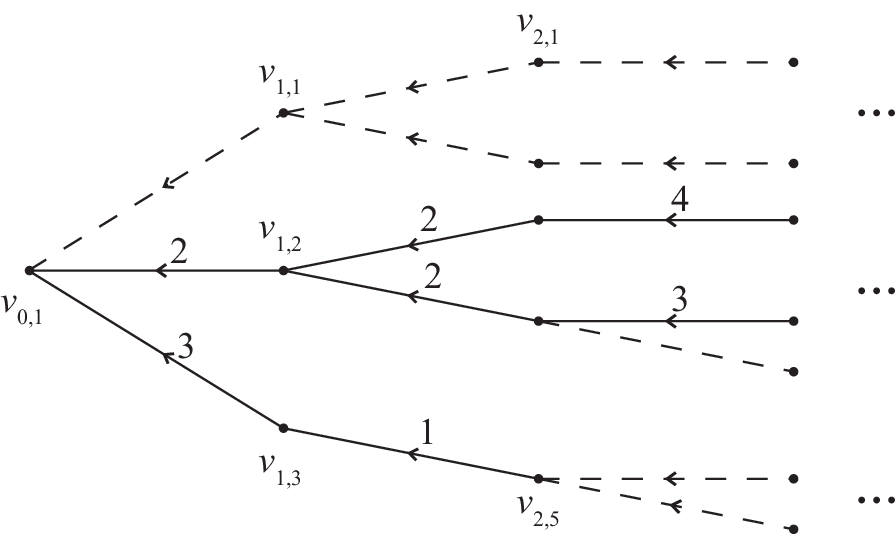}%
\caption{Example of a model tree $\Gamma$.}%
\label{Figure: Model tree}%
\end{figure}
%EndExpansion

Orient the edges of $\Gamma$ toward $v_{0,1}$ and give $\Gamma$ the path
length metric, with all edges assigned length $1$. We adopt the following
convention for denoting vertices, edges, and labels.

\begin{itemize}
\item[(ii)] A symbol $v_{i,j}$ indicates a vertex at a distance $i$ from the
root; vertices with initial index $i$ will be called the \emph{tier }$i$
\emph{vertices. }

\item[(iii)] For each $v_{i,j}$, with $i>0$, $e_{i,j}$ denotes the unique
oriented edge emanating from $v_{i,j}$ and $k_{i,j}=\mathcal{K}\left(
e_{i,j}\right)  $.
\end{itemize}

The null edges of $\Gamma$, together with their vertices, constitute a
(possibly empty) subgraph $\Omega$ of $\Gamma$ where each component contains a
unique vertex $v_{i,j}$ closest to $v_{0,1}$ in $\Gamma$. In this way,
$\Omega$ may be viewed as a rooted forest (a disjoint union of rooted
subtrees) $\left\{  \Omega_{i,j}\right\}  $, where an index $i,j$ indicates
that $v_{i,j}$ is its root. Of course, not every vertex of $\Gamma$ is the
root of a null subtree.

Two families of finite subtrees of $\Gamma$ also play a useful role: for each
integer $i\geq0$, let $\Gamma_{i}$ denote the $i$-neighborhood of $v_{0,1}$ in
$\Gamma$ and $\Gamma_{i}^{+}=\Gamma_{i}\cap\Gamma^{+}$.

\begin{remark}
The above definitions allow for the possibility $\Gamma=\left\{
v_{0,1}\right\}  $; but, except for that trivial case, $\Gamma$ must be
infinite. In fact, every edge of $\Gamma$ is contained in some infinite edge
path ray.
\end{remark}

\subsection{Model Base Spaces}

Next we describe the \emph{model base space} $X_{\Gamma}$ corresponding to a
model graph $\Gamma$; it will contain $\Gamma$ as a subcomplex.

\begin{itemize}
\item[(iv)] Attach an oriented edge $e_{0,1}^{\prime}$ to $\Gamma$ by
identifying each end to $v_{0,1}$; in a similar manner, attach an oriented
edge $e_{i,j}^{\prime}$ at each $v_{i,j}$ for which $k_{i,j}\neq0$. This
completes the 1-skeleton of $X_{\Gamma}$. For later use, let $S_{i,j}^{1}$
denote the oriented circle in $X_{\Gamma}^{\left(  1\right)  }$ that is the
image of $e_{i,j}^{\prime}$; it has natural base point $v_{i,j}$.

\item[(v)] For each $e_{i,j}$ with $k_{i,j}\neq0$, attach a 2-cell $d_{i,j}$
to $X_{\Gamma}^{\left(  1\right)  }$ as follows: beginning with $\left[
0,1\right]  \times\left[  0,1\right]  $, identify the top and bottom faces
with $e_{i,j}$, send the right face once around $e_{i,j}^{\prime}$, and the
left face $k_{i,j}$ times around $e_{i-1,j^{\prime}}^{\prime}$, where
$v_{i-1,j^{\prime}}$ is the terminal end of $e_{i,j}$. Notice that $d_{i,j}$
is the mapping cylinder of a canonical degree $k_{i,j}$ map of $S_{i,j}^{1}$
onto $S_{i-1,j^{\prime}}^{1}$. This completes the construction of $X_{\Gamma}%
$. See Figure \ref{Figure: Model base space}.
%TCIMACRO{\FRAME{ftbpFU}{3in}{2.0003in}{0pt}{\Qcb{Model base space $X_{\Gamma}$
%corresponding to Figure \ref{Figure: Model tree}.}}%
%{\Qlb{Figure: Model base space}}{coaxial-fig2.eps}%
%maintain-aspect-ratio TRUE;  display "USEDEF";  valid_file "F";  width 3in;
%height 2.0003in;  depth 0pt;  original-width 0pt;  original-height 0pt;
%cropleft "0";  croptop "1";  cropright "1";  cropbottom "0";
%filename 'Coaxial-fig2.eps';file-properties "XNPEU";}} }%
%BeginExpansion
\begin{figure}[ptb]%
\centering
\includegraphics[
height=2.0003in,
width=3in
]%
{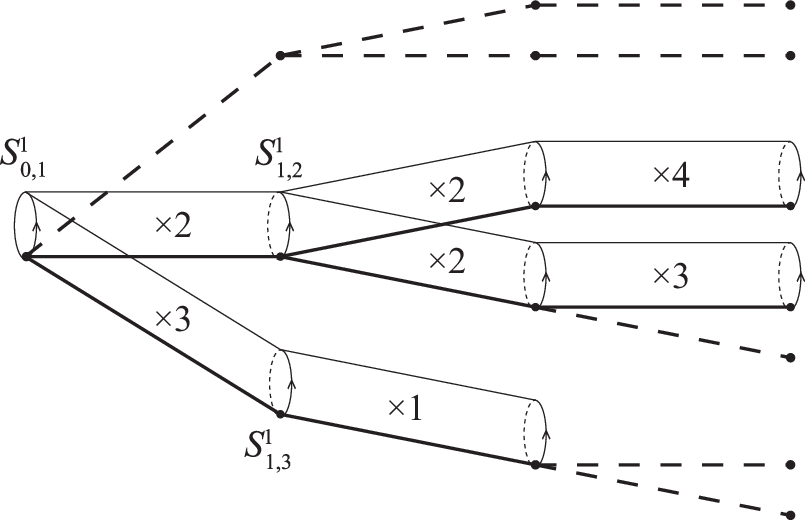}%
\caption{Model base space $X_{\Gamma}$ corresponding to Figure
\ref{Figure: Model tree}.}%
\label{Figure: Model base space}%
\end{figure}
%EndExpansion
Denote by $X_{\Gamma^{+}}$ the subcomplex made up of $\Gamma^{+}$ together
with all $e_{i,j}^{\prime}$ and all $d_{i,j}$; call $X_{\Gamma^{+}}$ the
\emph{positive subcomplex }of $X_{\Gamma}$. For each $i\geq0$, let
$X_{\Gamma_{i}}$ be the subcomplex of $X_{\Gamma}$ made up of $\Gamma_{i}$ and
all $e_{i,j}^{\prime}$ and $d_{i,j}$ attached to $\Gamma_{i}$. Define
$X_{\Gamma_{i}^{+}}$ similarly.\medskip
\end{itemize}

If we view the null edges of $\Gamma$ as mapping cylinders with singleton
domains, then $X_{\Gamma}$ is made up entirely of mapping cylinders. In fact,
if $W_{i}$ is the union of all $S_{i,j}^{1}$ and $v_{i,j}$ in the
$i^{\text{th}}$ tier, and $\omega_{i}:W_{i}\rightarrow W_{i-1}$ be the union
of maps taking the $S_{i,j}^{1}$ onto corresponding $S_{i-1,j^{\prime}}^{1}$
and $v_{i,j}$ to corresponding $v_{i-1,j^{\prime}}$, then $X_{\Gamma}$ is the
\emph{inverse mapping telescope} of the sequence
\[
S_{0,1}^{1}=W_{0}\overset{\omega_{1}}{\longleftarrow}W_{1}\overset{\omega
_{2}}{\longleftarrow}W_{2}\overset{\omega_{3}}{\longleftarrow}\cdots
\]
The natural deformation retraction of $X_{\Gamma}$ onto $S_{0,1}^{1}$, which
slides points along mapping telescope rays toward $S_{0,1}^{1}$, ends in a
retraction $\rho:X_{\Gamma}\rightarrow S_{0,1}^{1}$. See \cite{Gu13} for a
discussion of inverse mapping telescopes.

For the next stage of our construction, it will be useful to have a thorough
understanding of the point preimage $\rho^{-1}\left(  v_{0,1}\right)  $, which
consists of all mapping telescope rays, both infinite and finite, emanating
from $v_{0,1}$. (Finite mapping cylinder \textquotedblleft
rays\textquotedblright\ occur when an edge $e_{i,j}$ has label $k_{i,j}>1$ but
all edges of $\Gamma$ with terminus $v_{i,j}$ are null.) By subdividing these
rays in the obvious manner, with edges corresponding to the intersections with
individual mapping cylinders and vertices corresponding to intersections with
the $W_{i}$, $\rho^{-1}\left(  v_{0,1}\right)  $ becomes a tree $\Lambda$ with
root vertex $v_{0,1}$. This tree contains $\Gamma$, but potentially much more.
That is because each $d_{i,j}$, viewed as a mapping cylinder, contains
$k_{i,j}$ distinct cylinder lines ending at base vertex $v_{i-1,j^{\prime}}$.
Only one of those lines is an edge from $\Gamma$, but all are edges in
$\Lambda$.

We now describe $\Lambda$ as the union of inductively defined subtrees
$\Lambda_{1}\subseteq\Lambda_{2}\subseteq\cdots$.\medskip

\noindent\textbf{Step 1.} Beginning with $\Gamma_{1}$ as a building block,
expand it to $\Lambda_{1}$ as follows. Replace each $e_{1,j}$ with label
$k_{1,j}\neq0$ with a wedge of $k_{1,j}$ inwardly oriented edges having common
terminus $v_{0,1}$; color one edge from each such wedge black and the others
gray. View the black edge as the \textquotedblleft original\textquotedblright%
\ $e_{1,j}$ and its initial vertex as the original $v_{1,j}$; view the gray
edges and their initial vertices as \emph{Step 1 }\textquotedblleft
clones\textquotedblright. In addition, all null edges of $\Gamma_{1}$ are kept
as edges of $\Lambda_{1}$. As before, they are indicated by a black dashed
segment; the null edges do not get cloned. Call this finite tree, made up of
all black, gray, and dashed edges and their vertices, $\Lambda_{1}$. The black
and dashed edges form a copy of $\Gamma_{1}$ in $\Lambda_{1}$. The subtree
$\Lambda_{1}^{+}$, made up of black and gray edges and their vertices,
intersects $\Gamma_{1}$ in $\Gamma_{1}^{+}$. See Figure
\ref{Figure: Bass-Serre tree}.%
%TCIMACRO{\FRAME{ftbpFU}{3in}{2.0003in}{0pt}{\Qcb{$\rho^{-1}\left(
%v_{0,1}\right)  $ for the model base space in Figure
%\ref{Figure: Model base space}.}}{\Qlb{Figure: Bass-Serre tree}}%
%{coaxial-fig3.eps}{\special{ language "Scientific Word";  type "GRAPHIC";
%maintain-aspect-ratio TRUE;  display "USEDEF";  valid_file "F";  width 3in;
%height 2.0003in;  depth 0pt;  original-width 0pt;  original-height 0pt;
%cropleft "0";  croptop "1";  cropright "1";  cropbottom "0";
%filename 'Coaxial-fig3.eps';file-properties "XNPEU";}} }%
%BeginExpansion
\begin{figure}[ptb]%
\centering
\includegraphics[
height=2.0003in,
width=3in
]%
{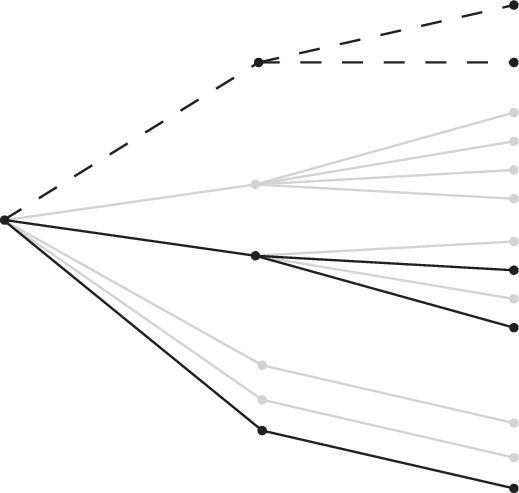}%
\caption{$\rho^{-1}\left(  v_{0,1}\right)  $ for the model base space in
Figure \ref{Figure: Model base space}.}%
\label{Figure: Bass-Serre tree}%
\end{figure}
%EndExpansion
\medskip

\noindent\textbf{Step 2.} To construct $\Lambda_{2}$, attach additional edges
and vertices to $\Lambda_{1}$ as follows. At the initial vertex $v_{1,j}$ of
each edge of $\Gamma_{1}^{+}\leq\Lambda$ (the black edges), attach a wedge of
$k_{2,j^{\prime}}$ edges for each non-null $e_{2,j^{\prime}}$ in $\Gamma_{2}$
terminating at $v_{1,j}$; color one edge from each wedge black and the others
gray. View the black edge as the original $e_{2,j^{\prime}}$ and its initial
vertex as $v_{2,j^{\prime}}$; the gray edges are Step 2 clones. In addition,
at each clone of each $v_{1,j}$ (the gray edges of $\Lambda_{1}$), place a
\textquotedblleft wedge of wedges\textquotedblright\ identical to the one just
attached at $v_{1,j}$, except that all of these edges are colored gray---they
are also Step 2 clones. Finally, at the initial vertex $v_{1,j}$ of
\emph{only} the black and dashed edges of $\Lambda_{1}$ add an incoming dashed
edge $e_{2,j^{\prime}}$ for each null $e_{2,j^{\prime}}$ in $\Gamma_{2}$
terminating at $v_{1,j}$. (As in Step 1, dashed edges do not get cloned.) Call
the resulting finite graph $\Lambda_{2}$. The black and dashed edges, and
their vertices form a copy of $\Gamma_{2}$ in $\Lambda_{2}$; meanwhile the
black and gray edges form a subtree $\Lambda_{2}^{+}$ which intersects
$\Gamma_{2}$ in $\Gamma_{2}^{+}$. Again see Figure
\ref{Figure: Bass-Serre tree}.\medskip

\noindent\textbf{Inductive steps.} Continue the above process inductively
outward to construct finite (colored trees) $\Lambda_{1}\subseteq\Lambda
_{2}\subseteq\Lambda_{3}\subseteq\cdots$ whose union is the tree $\Lambda
=\rho^{-1}\left(  v_{0,1}\right)  $, rooted at $v_{0,1}$ and containing
$\Gamma$ as a rooted subtree (the black and dashed edges). The subtree
consisting of all black and gray edges is denoted $\Lambda^{+}$; it intersects
$\Gamma$ in $\Gamma^{+}$.\medskip

\begin{remark}
\label{Remark: Bass-Serre trees}Experts will notice a similarity between the
above construction and a fundamental construction in Bass-Serre theory. At the
conclusion of this section, we will make a concrete connection between the
two.\medskip
\end{remark}

\subsection{Model $%
%TCIMACRO{\U{2124} }%
%BeginExpansion
\mathbb{Z}
%EndExpansion
$-spaces}

We now look to understand \emph{model }$%
%TCIMACRO{\U{2124} }%
%BeginExpansion
\mathbb{Z}
%EndExpansion
$\emph{-spaces }$\widetilde{X}_{\Gamma}$, which are the universal covers of
the $X_{\Gamma}$.

Let $q:\widetilde{X}_{\Gamma}\rightarrow X_{\Gamma}$ and $r:%
%TCIMACRO{\U{211d} }%
%BeginExpansion
\mathbb{R}
%EndExpansion
\rightarrow S_{0,1}^{1}$ be universal covering projections, where $S_{0,1}%
^{1}$ is viewed as the quotient of $%
%TCIMACRO{\U{2124} }%
%BeginExpansion
\mathbb{Z}
%EndExpansion
$ acting on $%
%TCIMACRO{\U{211d} }%
%BeginExpansion
\mathbb{R}
%EndExpansion
$ by unit translations. The lift $\widetilde{\rho}:\widetilde{X}_{\Gamma
}\rightarrow%
%TCIMACRO{\U{211d} }%
%BeginExpansion
\mathbb{R}
%EndExpansion
$ of $\rho:X\rightarrow S_{0,1}^{1}$ will play a useful role as a
\textquotedblleft height function\textquotedblright. For example, in the case
where $X_{\Gamma^{+}}\leq X_{\Gamma}$ is just $S_{0,1}^{1}$, $\widetilde{X}%
_{\Gamma}$ consists of a real line $\widetilde{S}_{0,1}^{1}$ taken
homeomorphically onto $%
%TCIMACRO{\U{211d} }%
%BeginExpansion
\mathbb{R}
%EndExpansion
$ by $\widetilde{\rho}$, together with a copy of $\Lambda$ (in this case the
same as $\Gamma$) attached at each integer height. The general case is
similar, in that $\widetilde{X}_{\Gamma}$ is made up of $\widetilde{X}%
_{\Gamma^{+}}$ along with trees attached at integer heights; but now both
$\widetilde{X}_{\Gamma^{+}}$ and the attachment pattern for the trees are more
complicated. Since $X_{\Gamma^{+}}$ is built entirely from cylinders of
nontrivial maps between circles, we can begin to understand $\widetilde{X}%
_{\Gamma^{+}}$ by looking at the universal cover of a single mapping cylinder.

The universal cover of the mapping cylinder $\mathcal{M}_{k}$ of a degree $k$
map $S^{1}\overset{\times k}{\longleftarrow}S^{1}$ can be realized as
$\widetilde{\mathcal{M}}_{k}=\Lambda\left(  k\right)  \times%
%TCIMACRO{\U{211d} }%
%BeginExpansion
\mathbb{R}
%EndExpansion
$, where $\Lambda\left(  k\right)  $ is a wedge of arcs with common end point
$a_{0}$ and distinct initial points $a_{1},\cdots,a_{k}$. Under the covering
projection, the preimage of the range circle is the line $\{a_{0}\}\times%
%TCIMACRO{\U{211d} }%
%BeginExpansion
\mathbb{R}
%EndExpansion
$ and the preimage of the domain circle is $\left\{  a_{1},\cdots
,a_{k}\right\}  \times%
%TCIMACRO{\U{211d} }%
%BeginExpansion
\mathbb{R}
%EndExpansion
$, one copy of $%
%TCIMACRO{\U{211d} }%
%BeginExpansion
\mathbb{R}
%EndExpansion
$ for each coset of $k%
%TCIMACRO{\U{2124} }%
%BeginExpansion
\mathbb{Z}
%EndExpansion
$ in $%
%TCIMACRO{\U{2124} }%
%BeginExpansion
\mathbb{Z}
%EndExpansion
$. The group of covering transformations is generated by the map $\sigma
_{k}\times t$, where $\sigma_{k}:\Lambda\left(  k\right)  \rightarrow
\Lambda\left(  k\right)  $ fixes $a_{0}$ and permutes the edges cyclically,
and $t\left(  r\right)  =r+1$.

Working inductively outward from\ $S_{0,1}^{1}$, and replicating the above
construction again and again, one sees that the subcomplex $\widetilde{X}%
_{\Gamma^{+}}$ may be identified with the product $\Lambda^{+}\times%
%TCIMACRO{\U{211d} }%
%BeginExpansion
\mathbb{R}
%EndExpansion
$, with the group of covering transformations being generated by a product map
$\sigma_{\infty}\times t$, where $\sigma_{\infty}:\Lambda^{+}\rightarrow
\Lambda^{+}$ is a homeomorphism that fixes $v_{0,1}$ and is determined by how
it permutes the ends of $\Lambda^{+}$, and $t\left(  r\right)  =r+1$.

\begin{remark}
The homeomorphism $\sigma_{\infty}:\Lambda^{+}\rightarrow\Lambda^{+}$ can be
built inductively from the various $\sigma_{k}$ described above. A more
algebraic description can be obtained from Bass-Serre theory, where
$\Lambda^{+}$ is viewed as the Bass-Serre tree corresponding to a graph of
groups interpretation of $\Gamma^{+}$ and $\sigma_{\infty}$ is the generator
of the corresponding action. See \S \ref{Section: Bass-Serre theory}.
\end{remark}

In situations where $X_{\Gamma}=X_{\Gamma^{+}}$ (an important special case),
the above provides a complete description of $\widetilde{X}_{\Gamma}$ as
$\Lambda^{+}\times%
%TCIMACRO{\U{211d} }%
%BeginExpansion
\mathbb{R}
%EndExpansion
$ with covering transformations generated by $\sigma_{\infty}\times t$. In
general, we must account for the portions of $\widetilde{X}_{\Gamma}$ lying
over $X_{\Gamma}-X_{\Gamma^{+}}$. With respect to the height function, those
portions lie entirely at integer levels, where $\widetilde{\rho}^{-1}\left(
n\right)  $ is a copy of $\Lambda$ intersecting $\Lambda^{+}\times%
%TCIMACRO{\U{211d} }%
%BeginExpansion
\mathbb{R}
%EndExpansion
$ in $\Lambda^{+}\times\{n\}$. At $n=0$, a copy of $\Lambda$ is glued to
$\Lambda^{+}\times%
%TCIMACRO{\U{211d} }%
%BeginExpansion
\mathbb{R}
%EndExpansion
$ by identifying the subgraph $\Lambda^{+}$ with $\Lambda^{+}\times\{0\}$. For
general height $n$, a copy of $\Lambda$ is attached along $\Lambda^{+}%
\times\{n\}$ by identifying $x\in\Lambda^{+}$ with $\,(\sigma_{\infty}%
^{n}\left(  x\right)  ,n)$.

To obtain a generator of the group of covering transformations on
$\widetilde{X}_{\Gamma}$, we must extend $\sigma_{\infty}\times t$ over the
copies of $\Lambda$ at the integral levels. Abusing notation slightly,
$\widetilde{X}_{\Gamma}$ is the quotient of the \emph{disjoint} union
$(\Lambda^{+}\times%
%TCIMACRO{\U{211d} }%
%BeginExpansion
\mathbb{R}
%EndExpansion
)\sqcup(\Lambda\times%
%TCIMACRO{\U{2124} }%
%BeginExpansion
\mathbb{Z}
%EndExpansion
)$, where $\left(  x,n\right)  $ in the second summand is identified
\emph{not} with $\left(  x,n\right)  $ in the first, but rather, with
$(\sigma_{\infty}^{n}\left(  x\right)  ,n)$ in the first summand. The
generator of the covering transformations is obtained by gluing the maps
$\sigma_{\infty}\times t:\Lambda^{+}\times%
%TCIMACRO{\U{211d} }%
%BeginExpansion
\mathbb{R}
%EndExpansion
\rightarrow\Lambda^{+}\times%
%TCIMACRO{\U{211d} }%
%BeginExpansion
\mathbb{R}
%EndExpansion
$ and $\operatorname*{id}\times t:\Lambda\times%
%TCIMACRO{\U{2124} }%
%BeginExpansion
\mathbb{Z}
%EndExpansion
\rightarrow\Lambda\times%
%TCIMACRO{\U{2124} }%
%BeginExpansion
\mathbb{Z}
%EndExpansion
$.

For easy reference, we assemble the key properties of $\widetilde{X}_{\Gamma}$
in a single proposition.

\begin{proposition}
\label{Prop: Understanding the universal cover of X}Let $\Gamma$ be a model
tree, $X_{\Gamma}$ its model space, and $q:\widetilde{X}_{\Gamma}\rightarrow
X_{\Gamma}$ the universal covering projection. Then $\widetilde{X}_{\Gamma}$
is a contractible 2-complex with $1,2$ or infinitely many ends. More
specifically, the pair $\Gamma^{+}\leq\Gamma$ (together with their labelings)
determine a pair of trees $\Lambda^{+}\leq\Lambda$, also rooted at $v_{0,1}$,
with $\Gamma^{+}\leq\Lambda^{+}$ and $\Gamma\leq\Lambda$ such that:

\begin{enumerate}
\item $\widetilde{X}_{\Gamma}$ is $2$-ended if and only if $\Gamma=\Gamma
^{+}=\left\{  v_{0,1}\right\}  $ (a single vertex). In that case
$\Lambda=\Lambda^{+}=\left\{  v_{0,1}\right\}  $ and $\widetilde{X}_{\Gamma
}\approx%
%TCIMACRO{\U{211d} }%
%BeginExpansion
\mathbb{R}
%EndExpansion
$, with the group of covering transformations generated by $t\left(  r\right)
=r+1$;

\item $\widetilde{X}_{\Gamma}$ is $1$-ended if and only if $\Gamma=\Gamma^{+}$
and the two are nontrivial (hence infinite). In that case $\Lambda=\Lambda
^{+}$ and $\widetilde{X}_{\Gamma}\approx\Lambda^{+}\times%
%TCIMACRO{\U{211d} }%
%BeginExpansion
\mathbb{R}
%EndExpansion
$, with the corresponding group of covering transformations generated by a
product of homeomorphisms $\sigma_{\infty}\times t:\Lambda^{+}\times%
%TCIMACRO{\U{211d} }%
%BeginExpansion
\mathbb{R}
%EndExpansion
\rightarrow\Lambda^{+}\times%
%TCIMACRO{\U{211d} }%
%BeginExpansion
\mathbb{R}
%EndExpansion
$, where $\sigma_{\infty}$ fixes the root of $\Lambda^{+}$ and $t\left(
r\right)  =r+1$;

\item $\widetilde{X}_{\Gamma}$ is infinite-ended if and only if $\Gamma
^{+}\lneq\Gamma$. In that case $\Omega=\overline{\Gamma-\Gamma^{+}}$ is a
nonempty forest $\left\{  \Omega_{i,j}\right\}  $ of infinite rooted trees,
and $\widetilde{X}_{\Gamma}$ is homeomorphic to $\Lambda^{+}\times%
%TCIMACRO{\U{211d} }%
%BeginExpansion
\mathbb{R}
%EndExpansion
$ together with a $%
%TCIMACRO{\U{2124} }%
%BeginExpansion
\mathbb{Z}
%EndExpansion
$-equivariant family $\left\{  n\Omega_{i,j}\right\}  _{n\in%
%TCIMACRO{\U{2124} }%
%BeginExpansion
\mathbb{Z}
%EndExpansion
}$ of copies of each $\Omega_{i,j}$ attached to $\Lambda^{+}\times%
%TCIMACRO{\U{211d} }%
%BeginExpansion
\mathbb{R}
%EndExpansion
$ at their roots. More specifically, a generator of the covering
transformations on $\widetilde{X}_{\Gamma}$ restricts to $\Lambda^{+}\times%
%TCIMACRO{\U{211d} }%
%BeginExpansion
\mathbb{R}
%EndExpansion
$ as a product of homeomorphisms $\sigma_{\infty}\times t:\Lambda^{+}\times%
%TCIMACRO{\U{211d} }%
%BeginExpansion
\mathbb{R}
%EndExpansion
\rightarrow\Lambda^{+}\times%
%TCIMACRO{\U{211d} }%
%BeginExpansion
\mathbb{R}
%EndExpansion
$, as described above, and $n\Omega_{i,j}$ is attached to $\Lambda^{+}\times%
%TCIMACRO{\U{211d} }%
%BeginExpansion
\mathbb{R}
%EndExpansion
$ by identifying its root to $\left(  \sigma_{\infty}^{n}\left(
v_{i,j}\right)  ,n\right)  $. The map $\sigma_{\infty}\times t$ extends to
$\widetilde{X}_{\Gamma}$ in the obvious way.
\end{enumerate}
\end{proposition}

\begin{remark}
\label{Remark. Subcases of infinite-ended Y}Case 3) of Proposition
\ref{Prop: Understanding the universal cover of X} can be split into subcases
resembling the 2- and 1-ended situations, respectively. \smallskip

\noindent Subcase a). When $\Gamma^{+}$ is finite, so is $\Lambda^{+}$, so a
collapse of $\Lambda^{+}$ onto its root vertex induces an equivariant proper
homotopy equivalence $f:\Gamma^{+}\times%
%TCIMACRO{\U{211d} }%
%BeginExpansion
\mathbb{R}
%EndExpansion
\rightarrow%
%TCIMACRO{\U{211d} }%
%BeginExpansion
\mathbb{R}
%EndExpansion
.$ If, at each integer $n$, we attach to $%
%TCIMACRO{\U{211d} }%
%BeginExpansion
\mathbb{R}
%EndExpansion
$ copies of the trees $n\Omega_{i,j}\subseteq\widetilde{X}_{\Gamma}$; then $f$
extends to an equivariant proper homotopy equivalence between $\widetilde{X}%
_{\Gamma}$ and the resulting locally finite graph comprised of $%
%TCIMACRO{\U{211d} }%
%BeginExpansion
\mathbb{R}
%EndExpansion
$ with trees attached at the integers.\smallskip

\noindent Subcase b).When $\Gamma^{+}$ is infinite, there is no obvious
simplification of $\widetilde{X}_{\Gamma}$, but an analogy with the 1-ended
case remains. In particular $\widetilde{X}_{\Gamma}$ contains a large
equivariant subcomplex identical to the 1-ended case, with the remainder of
$\widetilde{X}_{\Gamma}$ consisting of a discrete collection of
trees.\smallskip

\noindent Under either of the two subcases, $\widetilde{X}_{\Gamma}$ has
countably many ends, unless $\left\{  \Omega_{i,j}\right\}  $ contains a tree
with uncountably many ends. \medskip
\end{remark}

The usefulness of model spaces $X_{\Gamma}$ and $\widetilde{X}_{\Gamma}$ lies
in the simplicity of their topology at infinity. Of particular interest here
is their homotopy and homology data in dimensions 0 and 1.

\begin{proposition}
\label{Proposition: End properties of the model space}Let $\Gamma$ be a model
tree and $X_{\Gamma}$ the corresponding model space. Then the inclusion map
$\Gamma\hookrightarrow X_{\Gamma}$ is a proper 1-equivalence, thereby inducing
a bijection between ends. If $r$ is an edge path ray in $\Gamma$ beginning at
$v_{0,1}$, then $\operatorname*{pro}$-$\pi_{1}\left(  X_{\Gamma},r\right)  $
can be represented by the inverse sequence%
\[%
%TCIMACRO{\U{2124} }%
%BeginExpansion
\mathbb{Z}
%EndExpansion
\overset{\times k_{1,j_{1}}}{\longleftarrow}%
%TCIMACRO{\U{2124} }%
%BeginExpansion
\mathbb{Z}
%EndExpansion
\overset{\times k_{2,j_{2}}}{\longleftarrow}%
%TCIMACRO{\U{2124} }%
%BeginExpansion
\mathbb{Z}
%EndExpansion
\overset{\times k_{2,j_{3}}}{\longleftarrow}\cdots
\]
where the $k_{i,j_{i}}$ are the labels on the edges that comprise $r$.\medskip
\end{proposition}

Of greater interest is the end behavior of the model $%
%TCIMACRO{\U{2124} }%
%BeginExpansion
\mathbb{Z}
%EndExpansion
$-spaces.\medskip

\begin{proposition}
\label{Prop: Model Z-space}Let $\Gamma$ be a model tree, $X_{\Gamma}$ and
$\widetilde{X}_{\Gamma}$ the corresponding model $%
%TCIMACRO{\U{2124} }%
%BeginExpansion
\mathbb{Z}
%EndExpansion
$-space. As noted in Proposition
\ref{Prop: Understanding the universal cover of X}, $\widetilde{X}_{\Gamma}$
is 1-, 2-, or infinite-ended. Moreover,

\begin{enumerate}
\item \label{Prop: Model Z-space case 1}If $\widetilde{X}_{\Gamma}$ is
2-ended, both ends are simply connected and the $%
%TCIMACRO{\U{2124} }%
%BeginExpansion
\mathbb{Z}
%EndExpansion
$-action fixes those ends;

\item \label{Prop: Model Z-space case 2}If $\widetilde{X}_{\Gamma}$ is
1-ended, that end is semistable and $\operatorname*{pro}$-$\pi_{1}\left(
\widetilde{X}_{\Gamma},r\right)  $ can be represented by an inverse sequence
of surjections between finitely generated free groups%
\[
F_{1}\twoheadleftarrow F_{2}\twoheadleftarrow F_{3}\twoheadleftarrow\cdots
\]
and $\operatorname*{pro}$-$H_{1}\left(  \widetilde{X}_{\Gamma};%
%TCIMACRO{\U{2124} }%
%BeginExpansion
\mathbb{Z}
%EndExpansion
\right)  $ can be represented by an inverse sequence of surjections between
finitely generated free abelian groups%
\[%
%TCIMACRO{\U{2124} }%
%BeginExpansion
\mathbb{Z}
%EndExpansion
^{n_{1}}\twoheadleftarrow%
%TCIMACRO{\U{2124} }%
%BeginExpansion
\mathbb{Z}
%EndExpansion
^{n_{2}}\twoheadleftarrow%
%TCIMACRO{\U{2124} }%
%BeginExpansion
\mathbb{Z}
%EndExpansion
^{n_{3}}\twoheadleftarrow\cdots.
\]

\item \label{Prop: Model Z-space case 3}If $\widetilde{X}_{\Gamma}$ is
infinite-ended, the $%
%TCIMACRO{\U{2124} }%
%BeginExpansion
\mathbb{Z}
%EndExpansion
$-action fixes precisely one or two ends with the others having trivial
stabilizers. All non-fixed ends are simply-connected. If two ends are fixed,
those ends are simply connected as well. If just one end is fixed, that end is
semistable with $\operatorname*{pro}$-$\pi_{1}\left(  \widetilde{X}_{\Gamma
},r\right)  $ representable by an inverse sequence like the one described in
Assertion (\ref{Prop: Model Z-space case 2}). Similarly, $\operatorname*{pro}%
$-$H_{1}\left(  \widetilde{X}_{\Gamma};%
%TCIMACRO{\U{2124} }%
%BeginExpansion
\mathbb{Z}
%EndExpansion
\right)  $ is representable by a sequence like the one found in Assertion
(\ref{Prop: Model Z-space case 2}), with all nontrivial contributions coming
from the fixed end.
\end{enumerate}
\end{proposition}

\begin{proof}
The only assertions not immediate from Proposition
\ref{Prop: Understanding the universal cover of X} are the representations of
$\operatorname*{pro}$-$\pi_{1}\left(  \widetilde{X}_{\Gamma},r\right)  $ and
$\operatorname*{pro}$-$H_{1}\left(  \widetilde{X}_{\Gamma};%
%TCIMACRO{\U{2124} }%
%BeginExpansion
\mathbb{Z}
%EndExpansion
\right)  $. Let us first address the 1-ended case where, by Proposition
\ref{Prop: Understanding the universal cover of X}, $\widetilde{X}_{\Gamma}$
may be identified with $\Lambda^{+}\times%
%TCIMACRO{\U{211d} }%
%BeginExpansion
\mathbb{R}
%EndExpansion
$, with $\Lambda^{+}$ an infinite leafless tree rooted at $v_{0,1}$. Let $r=$
$v_{0,1}\times\lbrack0,\infty)$ be the base ray, and $N_{1}\supseteq
N_{2}\supseteq\cdots$ the cofinal sequence of neighborhoods of infinity, where
$N_{i}=\Lambda^{+}\times%
%TCIMACRO{\U{211d} }%
%BeginExpansion
\mathbb{R}
%EndExpansion
-[\mathring{\Lambda}_{i}^{+}\times\left(  -i,i\right)  ]$. Here $\mathring
{\Lambda}_{i}^{+}$ is the open $i$-ball in $\Lambda_{i}$ centered at $v_{0,1}%
$. It is easy to see that $N_{i}$ deformation retracts onto its frontier in
$\Lambda^{+}\times%
%TCIMACRO{\U{211d} }%
%BeginExpansion
\mathbb{R}
%EndExpansion
$,
\[
\operatorname*{Fr}\nolimits_{\Lambda^{+}\times%
%TCIMACRO{\U{211d} }%
%BeginExpansion
\mathbb{R}
%EndExpansion
}N_{i}=\Lambda_{i}^{+}\times\left\{  -i,i\right\}  \cup(\operatorname*{Fr}%
\nolimits_{\Lambda^{+}}\Lambda_{i}^{+}\times\left[  -i,i\right]  )
\]
where $\operatorname*{Fr}_{\Lambda^{+}}\Lambda_{i}^{+}$ is the set of vertices
in $\Lambda^{+}$ at a distance $i$ from $v_{0,1}$. By squeezing $\Lambda
_{i}^{+}\times\left\{  -i\right\}  $ and $\Lambda_{i}^{+}\times\left\{
i\right\}  $ to points, $\operatorname*{Fr}\nolimits_{\Lambda^{+}\times%
%TCIMACRO{\U{211d} }%
%BeginExpansion
\mathbb{R}
%EndExpansion
}N_{i}$ is seen to be homotopy equivalent to the suspension of
$\operatorname*{Fr}_{\Lambda^{+}}\Lambda_{i}^{+}$, a space whose fundamental
group is free of rank $\left\vert \operatorname*{Fr}_{\Lambda^{+}}\Lambda
_{i}^{+}\right\vert -1$; call that group $F_{i}$. To complete Assertion
(\ref{Prop: Model Z-space case 2}), it remains to show that bonding maps
$F_{i}\longleftarrow F_{i+1}$ are surjective. Since $\Lambda^{+}$ has no
leaves, the collapse of $\Lambda_{i+1}^{+}$ onto $\Lambda_{i}^{+}$ restricts
to a surjection of $\operatorname*{Fr}\nolimits_{\Lambda^{+}}\Lambda_{i+1}%
^{+}$ onto $\operatorname*{Fr}\nolimits_{\Lambda^{+}}\Lambda_{i}^{+}$, which
can be suspended to get a map making the following diagram commute up to
homotopy.
\[%
\begin{tabular}
[c]{ccc}%
$N_{i}$ & $\hookleftarrow$ & $N_{i}$\\
$\uparrow\simeq$ &  & $\uparrow\simeq$\\
$\operatorname*{susp}\left(  \operatorname*{Fr}_{\Lambda^{+}}\Lambda_{i}%
^{+}\right)  $ & $\overset{s_{i+1}}{\longleftarrow}$ & $\operatorname*{susp}%
\left(  \operatorname*{Fr}_{\Lambda^{+}}\Lambda_{i+1}^{+}\right)  $%
\end{tabular}
\ \ \ \ \
\]
Surjectivity of the induced maps on fundamental groups is now clear.

To obtain an equivalent representation of $\operatorname*{pro}$-$\pi
_{1}\left(  \widetilde{X}_{\Gamma},r\right)  $ in the infinite-ended case with
a single fixed end, note that the fixed end can be represented by a sequence
$M_{1}\supseteq M_{2}\supseteq\cdots$ of components of neighborhoods of
infinity where each $M_{i}$ is homeomorphic to an $N_{i}$ from the previous
case, with a countable collection of locally finite trees attached at a
discrete collection of points. Since $M_{i}$ deformation retracts onto $N_{i}%
$, the above calculations are still valid.

The proposed representations of $\operatorname*{pro}$-$H_{1}\left(
\widetilde{X}_{\Gamma};%
%TCIMACRO{\U{2124} }%
%BeginExpansion
\mathbb{Z}
%EndExpansion
\right)  $ follow easily.
\end{proof}

\begin{remark}
If desired, more detail on the representations of $\operatorname*{pro}$%
-$\pi_{1}\left(  \widetilde{X}_{\Gamma},r\right)  $ and $\operatorname*{pro}%
$-$H_{1}\left(  \widetilde{X}_{\Gamma};%
%TCIMACRO{\U{2124} }%
%BeginExpansion
\mathbb{Z}
%EndExpansion
\right)  $ can be obtained; for example, formulas for the bonding maps, and a
description of the induced $%
%TCIMACRO{\U{2124} }%
%BeginExpansion
\mathbb{Z}
%EndExpansion
$-action on the inverse sequences can be deduced from the above analysis.
\end{remark}

\subsection{Reductions of model
spaces\label{Subsection: Reductions of model spaces}}

We close this section by describing a \textquotedblleft
reduction\textquotedblright\ procedure that can be applied to a model tree and
passed along to its resulting model spaces. Beginning with a model tree
$\Gamma$ and a pair of integers $0\leq i<j$, the \emph{elementary }%
$[i,j]$\emph{-reduction} is accomplished by removing all edges in
$\overline{\Gamma_{j}-\Gamma_{i}}$, then putting in a single edge from each
tier $j$ vertex $v_{j,r}$ to the unique tier $i$ vertex $v_{i,s}$ on the
reduced edge path connecting $v_{j,r}$ to the root vertex $v_{0,1}$. The label
on that new edge is the product of the labels on the edge path in $\Gamma$
connecting $v_{j,r}$ to $v_{i,s}$. If the new tree is denoted $\Gamma^{\prime
}$ then, topologically, $\Gamma^{\prime}$ is obtained from $\Gamma$ by
crushing each component of $\overline{\Gamma_{j-1}-\Gamma_{i}}$ to a point.

The difference between $X_{\Gamma}$ and $X_{\Gamma^{\prime}}$ is easy to
discern. Remove from $X_{\Gamma}$ the interior of $\overline{X_{\Gamma_{j}%
}-X_{\Gamma_{i}}}$; then, for each tier $j$ circle $S_{j,r}^{1}$ replace the
\textquotedblleft path of mapping cylinders\textquotedblright\ in $X_{\Gamma}$
from $S_{j,r}^{1}$ to $S_{i,s}^{1}$ with a single mapping cylinder whose map
is the composition of the maps along that path. For a \textquotedblleft
naked\textquotedblright\ tier $j$ vertex, simply insert a naked edge
connecting it to the corresponding tier $i$ vertex. A standard fact about
mapping cylinders is that, for a composition $A\overset{f}{\longrightarrow
}B\overset{g}{\longrightarrow}C$, the mapping cylinder $\operatorname*{Map}%
\left(  gf\right)  $ of the composition is homotopy equivalent rel $A\cup C$
to the union $\operatorname*{Map}\left(  f\right)  \cup_{B}\operatorname*{Map}%
\left(  g\right)  $ of mapping cylinders. Applying this fact repeatedly, one
obtain a proper homotopy equivalence, fixed outside the interior of
$\overline{X_{\Gamma_{j}}-X_{\Gamma_{i}}}$, between $X_{\Gamma}$ and
$X_{\Gamma^{\prime}}$.

A \emph{reduction} of $\Gamma$ is obtained by performing the above procedure
over a, possibly infinite, sequence of closed intervals $\left\{  \left[
i_{k},j_{k}\right]  \right\}  $ with $j_{k}\leq i_{k+1}$ for all $k$. By
applying the above procedure repeatedly, and then lifting to universal covers,
we obtain the following useful fact.

\begin{proposition}
\label{Prop: Result of a reduction}Let $\Gamma^{\prime}$ be a model tree
obtained by reduction of a model tree $\Gamma$. Then the model base spaces
$X_{\Gamma^{\prime}}$ and $X_{\Gamma}$ are proper homotopy equivalent and the
model $%
%TCIMACRO{\U{2124} }%
%BeginExpansion
\mathbb{Z}
%EndExpansion
$-spaces $\widetilde{X}_{\Gamma^{\prime}}$ and $\widetilde{X}_{\Gamma}$ are
equivariantly proper homotopy equivalent.
\end{proposition}

\begin{example}
The proper homotopy equivalence discussed in subcase a) of Remark
\ref{Remark. Subcases of infinite-ended Y} can now be viewed as the result of
a reduction. Choose $j$ so large that $\Gamma^{+}\subseteq\Gamma_{j}$ and
perform the elementary $[i,j]$-reduction.
\end{example}

\section{Connections to Bass-Serre theory\label{Section: Bass-Serre theory}}

This section is a brief diversion. Bass-Serre theory is not needed for the
purposes of this paper, but for those with a previous understanding of that
topic, the connection can make some of our constructions easier to follow.

Beginning with a model tree $\Gamma$, create a \emph{graph of groups} as
follows: place a copy of $%
%TCIMACRO{\U{2124} }%
%BeginExpansion
\mathbb{Z}
%EndExpansion
$ on each vertex and each edge of $\Gamma^{+}$ and a trivial group $0$ on the
vertices and edges in $\Gamma-\Gamma^{+}$; then interpret the labels $k_{i,j}$
as multiplication homomorphisms. The result is an elaborate graph of groups
decomposition of $%
%TCIMACRO{\U{2124} }%
%BeginExpansion
\mathbb{Z}
%EndExpansion
$, where the copy of $%
%TCIMACRO{\U{2124} }%
%BeginExpansion
\mathbb{Z}
%EndExpansion
$ at the root vertex includes isomorphically into the fundamental group of the
graph of groups. (All homomorphisms on reversed edges are identities.) The
model space $X_{\Gamma}$ is the corresponding\emph{ total space} for $\Gamma$,
as described in \cite{ScWa77} and \cite[Ch.6]{Ge08}. The subgraph $\Gamma^{+}$
determines a simpler graph of groups decomposition of $%
%TCIMACRO{\U{2124} }%
%BeginExpansion
\mathbb{Z}
%EndExpansion
$ that is consistent with $\Gamma$ and has total space $X_{\Gamma}%
^{+}\subseteq X_{\Gamma}$. The tree $\Lambda^{+}$ constructed above is the
Bass-Serre tree corresponding to $\Gamma^{+}$ and $\sigma_{\infty}$ is a
generator of the corresponding action. See \cite[Ch.I, \S 4.5]{Se80}.

The Bass-Serre tree $\Lambda^{\ast}$ for the full graph of groups $\Gamma$
does not play a direct role here, but it is lurking in the background. One may
expand $\Lambda$ to $\Lambda^{\ast}$ as follows: Viewing $\Gamma$ as a subset
of $\Lambda$, replace each subtree $\Omega_{i,j}\leq\Gamma$ with a countably
infinite wedge of copies of $\Omega_{i,j}$, all joined at the root vertex
$v_{i,j}$ of $\Omega_{i,j}$. Designate one copy as the original $\Omega_{i,j}$
and the rest as clones. Then, at each clone of $v_{i,j}$ in $\Lambda$, attach
another infinite wedge of copies of $\Omega_{i,j}$, all viewed as clones. The
need for countably infinite collections is because the group at $v_{i,j}$ is $%
%TCIMACRO{\U{2124} }%
%BeginExpansion
\mathbb{Z}
%EndExpansion
$ while all incoming edge groups are trivial, and thus have countably infinite
index in the vertex group at $v_{i,j}$.

\section{Associating models to $%
%TCIMACRO{\U{2124} }%
%BeginExpansion
\mathbb{Z}
%EndExpansion
$-actions\label{Section: Associating models to proper Z-actions}}

We return to the primary objects of interest---simply connected, locally
compact ANRs admitting $%
%TCIMACRO{\U{2124} }%
%BeginExpansion
\mathbb{Z}
%EndExpansion
$-actions by covering transformations. Observations from Sections
\ref{Subsection: Proper homotopy equivalences and n-equivalences} and
\ref{Section: Coaxial and Strongly Coaxial homeomorphisms} allow us to focus
on strongly locally finite CW complexes (or even locally finite polyhedra)
admitting such $%
%TCIMACRO{\U{2124} }%
%BeginExpansion
\mathbb{Z}
%EndExpansion
$-actions. In this section, we prove the primary technical results of this
paper. At the conclusion, we will have obtained the following:

\begin{theorem}
For $Y$ a simply connected, locally compact ANR, and $j:Y\rightarrow Y$ a
homeomorphism generating an action by covering transformations with
$J\equiv\left\langle j\right\rangle \cong%
%TCIMACRO{\U{2124} }%
%BeginExpansion
\mathbb{Z}
%EndExpansion
$, there is a corresponding model tree $\Gamma$ so that

\begin{enumerate}
\item $\widetilde{X}_{\Gamma}$ is $%
%TCIMACRO{\U{2124} }%
%BeginExpansion
\mathbb{Z}
%EndExpansion
$--equivariantly properly $1$-equivalent to $Y$,

\item If $j$ is strongly coaxial, $J\backslash Y$ is properly 2-equivalent to
$X_{\Gamma}$; hence $Y$ is $%
%TCIMACRO{\U{2124} }%
%BeginExpansion
\mathbb{Z}
%EndExpansion
$-equivariantly properly 2-equivalent to $\widetilde{X}_{\Gamma}$.

\item If $j$ is coaxial, $Y$ is properly 2-equivalent to $\widetilde{X}%
_{\Gamma}$ via proper 2-equivalences that are $%
%TCIMACRO{\U{2124} }%
%BeginExpansion
\mathbb{Z}
%EndExpansion
$-equivariant on 1-skeleta.\medskip
\end{enumerate}
\end{theorem}

By our work in Sections
\ref{Subsection: Proper homotopy equivalences and n-equivalences} and
\ref{Section: Coaxial and Strongly Coaxial homeomorphisms}, it is enough to
consider the case where $Y$ is a simply connected, strongly locally finite CW
complex, and $j:Y\rightarrow Y$ is a cellular homeomorphism generating an
action by covering transformations with $J\equiv\left\langle j\right\rangle
\cong%
%TCIMACRO{\U{2124} }%
%BeginExpansion
\mathbb{Z}
%EndExpansion
$. Our first goal is to associate a model tree $\Gamma$ to this action. Begin
by choosing a nested cofinal sequence $J\backslash Y=N_{0}\supseteq
N_{1}\supseteq N_{2}\supseteq\cdots$ of subcomplex neighborhoods of infinity
in $J\backslash Y$. By discarding compact components, we may assume that each
of the (finitely many) components $\left\{  N_{i,j}\right\}  _{j=1}^{r_{i}}$
of each $N_{i}$ is unbounded.

Choose an oriented edge path loop $\alpha_{0,1}$ in $N_{0}=J\backslash Y$ that
generates $H_{1}\left(  J\backslash Y\right)  \cong%
%TCIMACRO{\U{2124} }%
%BeginExpansion
\mathbb{Z}
%EndExpansion
$. For each component $N_{i,j}$ of each $N_{i}$ consider the inclusion induced
map $H_{1}(J\backslash Y)\overset{i_{\ast}}{\longleftarrow}H_{1}\left(
N_{i,j}\right)  $. (All homology is with $%
%TCIMACRO{\U{2124} }%
%BeginExpansion
\mathbb{Z}
%EndExpansion
$-coefficients.) If the map is nontrivial, let $n_{i,j}$ be the index of
$i_{\ast}\left(  H_{1}\left(  N_{i,j}\right)  \right)  $ in $H_{1}(J\backslash
Y)$, and choose an oriented edge path loop $\alpha_{i,j}$ in $N_{i,j}$ taken
to $n_{i,j}\alpha_{0,1}$ by $i_{\ast}$; if it is trivial, let $n_{i,j}=0$ and
let $\alpha_{i,j}$ be a constant edge path loop in $N_{i,j}$.

\begin{remark}
Use of homology rather than fundamental group, in defining $n_{i,j}$ and
$\alpha_{i,j}$, allows us to avoid base point technicalities without loss of
any essential information.
\end{remark}

Let $K_{0}$ be a finite connected subcomplex of $J\backslash Y$ that contains
$\alpha_{0,1}$, and for each $i>0$, let $K_{i}$ be a finite connected
subcomplex of $J\backslash Y$ chosen sufficiently large so that

\begin{enumerate}
\item $\overline{J\backslash Y-N_{i}}\subseteq K_{i}$,

\item \label{item 2}for every pair of vertices in the frontier of a component
$N_{i,j}$ of $N_{i}$, $K_{i}$ contains an edge path in $N_{i,j}$ connecting
them, and

\item $K_{i}$ contains each loop in the collection $\left\{  \alpha
_{i,j}\right\}  _{j=1}^{r_{i}}$.
\end{enumerate}

By passing to a subsequence and relabeling, we may assume that $N_{i+1}%
\subseteq J\backslash Y-K_{i}$ for all $i$. Let $L_{i}=N_{i}\cap K_{i}$ and
$M_{i}=N_{i}\cap K_{i+1}$; then $M_{i}$ is a finite complex containing
disjoint subcomplexes $L_{i}$ and $L_{i+1}$, and $M_{i}\cap M_{i+1}=L_{i+1}$.
For each component $N_{i,j}$ of $N_{i}$, let $L_{i,j}=N_{i,j}\cap L_{i}$ and
$M_{i,j}=N_{i,j}\cap M_{i}$. By connectedness of $K_{i}$ and $N_{i,j}$, along
with property \ref{item 2}, each $L_{i,j}$ and $M_{i,j}$ is connected;
moreover, $M_{i,j}$ contains a component $L_{i+1,k}$ of $L_{i+1}$ if and only
if $N_{i,j}$ contains $N_{i+1,k}$. See Figure
\ref{Figure: Decomposition of Y mod J}.%
%TCIMACRO{\FRAME{ftbpFU}{3in}{2.0003in}{0pt}{\Qcb{Decomposition of $J\backslash
%Y$ into subcomplexes.}}{\Qlb{Figure: Decomposition of Y mod J}}%
%{coaxial-fig4.eps}{\special{ language "Scientific Word";  type "GRAPHIC";
%maintain-aspect-ratio TRUE;  display "USEDEF";  valid_file "F";  width 3in;
%height 2.0003in;  depth 0pt;  original-width 0pt;  original-height 0pt;
%cropleft "0";  croptop "1";  cropright "1";  cropbottom "0";
%filename 'Coaxial-fig4.eps';file-properties "XNPEU";}} }%
%BeginExpansion
\begin{figure}[ptb]%
\centering
\includegraphics[
height=2.0003in,
width=3in
]%
{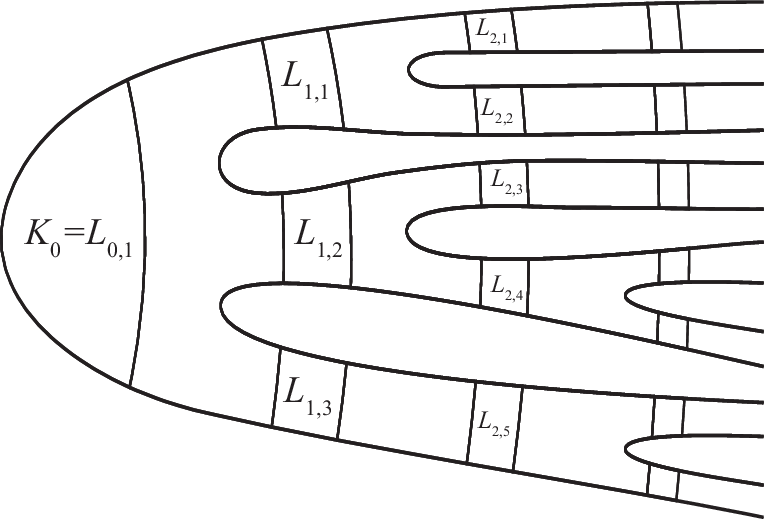}%
\caption{Decomposition of $J\backslash Y$ into subcomplexes.}%
\label{Figure: Decomposition of Y mod J}%
\end{figure}
%EndExpansion

Let $\Gamma$ be the rooted tree with a vertex $v_{i,j}$ for each $L_{i,j}$ and
an edge between $v_{i,p}$ and $v_{i+1,q}$ whenever $L_{i+1,q}\subseteq
M_{i,p}$ (equivalently $N_{i+1,q}\subseteq N_{i,p}$). The root vertex
$v_{0,1}$ corresponds to the single component $L_{0,1}$ of $L_{0}=K_{0}$ lying
in $N_{0,1}=N_{0}=J\backslash Y$. Since the $N_{i}$ have no compact
components, $\Gamma$ has no valence $1$ vertices.

Orient the edges of $\Gamma$ in the direction of $v_{0,1}$ and for each
$v_{i,j}$ with $i>0$, let $e_{i,j}$ denote the unique oriented edge emanating
from $v_{i,j}$. Label each $e_{i,j}$ with an integer $k_{i,j}$ as follows. For
the edges $e_{1,j}$ terminating at the root, let $k_{1,j}=n_{1,j}$. For $i>1$,
$k_{i,j}=0$ if $n_{i,j}=0$; otherwise let $k_{i,j}=\frac{n_{i,j}%
}{n_{i-1,j^{\prime}}}$, where $N_{i-1,j^{\prime}}$ is the unique component of
$N_{i-1}$ containing $N_{i,j}$. Since the map $H_{1}(J\backslash
Y)\overset{i_{\ast}}{\longleftarrow}H_{1}\left(  N_{i,j}\right)  $ used to
define $n_{i,j}$ factors through $H_{1}\left(  N_{i-1,j^{\prime}}\right)  $,
$k_{i,j}$ is an integer; moreover, for any $v_{i,j}$ the integer $n_{i,j}$ can
be recovered by multiplying the labels on the edge path connecting $v_{i,j}$
to $v_{0,1}$. Note that $\Gamma$ satisfies all conditions laid out in
\S \ref{Section: Model spaces} for a model tree; therefore, all definitions,
notation, and subsequent constructions from that section can be carried forward.

The tree $\Gamma$ is a good model for $\mathcal{E}\emph{nds}\left(
J\backslash Y\right)  $. Indeed, repeated application of the Tietze extension
theorem produces a proper 1-equivalence from $J\backslash Y$ to $\Gamma$.
Unfortunately, that map is of limited use: first, it has no chance of
providing information about higher-dimensional end invariants; and second, it
tells us nothing about the space $Y$, which is our primary interest. To
address those problems we construct a more delicate, map $f:J\backslash
Y\rightarrow X_{\Gamma}$ which incorporates some higher-dimensional
information and lifts to a map $\widetilde{f}:Y\rightarrow\widetilde{X}%
_{\Gamma}$. \medskip

Let $r:X_{\Gamma}\rightarrow\Gamma$ be the retraction sending each circle
$S_{i,j}^{1}$ onto $v_{i,j}$, and more generally, squashes each mapping
cylinder $d_{i,j}$ onto $e_{i,j}$ in a level-preserving manner, with point
preimages being circles. Notice that $r^{-1}\left(  \Gamma^{+}\right)
=X_{\Gamma^{+}}$. For each $i$, let $X_{\Gamma_{i}}=r^{-1}\left(  \Gamma
_{i}\right)  $ and $Q_{i}=\overline{X_{\Gamma}-X_{\Gamma_{i}}}$. Then
$X_{\Gamma_{0}}\subseteq X_{\Gamma_{1}}\subseteq X_{\Gamma_{2}}\subseteq
\cdots$ is a filtration of $X_{\Gamma}$ by finite subcomplexes, and
$X_{\Gamma}=Q_{0}\supseteq Q_{1}\supseteq Q_{2}\supseteq\cdots$ is a cofinal
sequence of subcomplex neighborhoods of infinity. For each $i$, let
$P_{i}=Q_{i}\cap X_{\Gamma_{i+1}}=r^{-1}\left(  \overline{\Gamma_{i+1}%
-\Gamma_{i}}\right)  $, a finite subcomplex consisting of $\overline
{\Gamma_{i+1}-\Gamma_{i}}$ with mapping cylinders attached along the non-null edges.

By construction, there is a one-to-one correspondence between the sequences of
neighborhoods of infinity $\left\{  N_{i}\right\}  $ and $\left\{
Q_{i}\right\}  $ so that the components $\left\{  N_{i,j}\right\}  $ of
$N_{i}$ are in one-to-one correspondence with the components $\left\{
Q_{i,j}\right\}  $ of $Q_{i}$. Moreover, for each component $M_{i,j}$ of
$M_{i}$, which contains a connected subcomplex $L_{i,j}$ on its
\textquotedblleft left-hand side\textquotedblright\ and a disjoint collection
of similar subcomplexes $\{L_{i+1,j^{\prime}}\}$ on its \textquotedblleft
right-hand side\textquotedblright, the corresponding component $P_{i,j}$ of
$P_{i}$ has a left-hand side consisting of a circle $S_{i,j}^{1}$ or vertex
$v_{i,j}$ and a right-hand side made up of circles and vertices, labeled
$S_{i+1,j^{\prime}}^{1}$ or $v_{i+1,j^{\prime}}$ (one for each subcomplex
$L_{i+1,j^{\prime}}$ in $M_{i,j}$). If some of right-hand components of
$P_{i,j}$ are circles, the left-hand side must be a circle, and $P_{i,j}$ is
made up of a union of mapping cylinders of degree $k_{i+1,j^{\prime}}$ maps
$S_{i,j}^{1}\leftarrow S_{i+1,j^{\prime}\text{ }}^{1}$ (one for each
right-hand circle) intersecting in a common range circle and \textquotedblleft
naked edges\textquotedblright\ connecting the isolated vertices of the
right-hand side to $v_{i,j}$ on the left-hand side. The map $f:J\backslash
Y\rightarrow X$ will be most easily understood from its restrictions
$f_{i,j}:M_{i,j}\rightarrow P_{i,j}$. See Figure
\ref{Figure: Constructon of fij}.%
%TCIMACRO{\FRAME{ftbpFU}{4.7859in}{1.9951in}{0pt}{\Qcb{A building block of
%$f:J\backslash Y\rightarrow X_{\Gamma}$.}}{\Qlb{Figure: Constructon of fij}%
%}{coaxial-fig5.eps}{\special{ language "Scientific Word";  type "GRAPHIC";
%display "USEDEF";  valid_file "F";  width 4.7859in;  height 1.9951in;
%depth 0pt;  original-width 0pt;  original-height 0pt;  cropleft "0";
%croptop "1";  cropright "1.0032";  cropbottom "0";
%filename 'Coaxial-fig5.eps';file-properties "XNPEU";}} }%
%BeginExpansion
\begin{figure}[ptb]%
\centering
\includegraphics[
trim=0.000000in 0.000000in -0.015315in 0.000000in,
height=1.9951in,
width=4.7859in
]%
{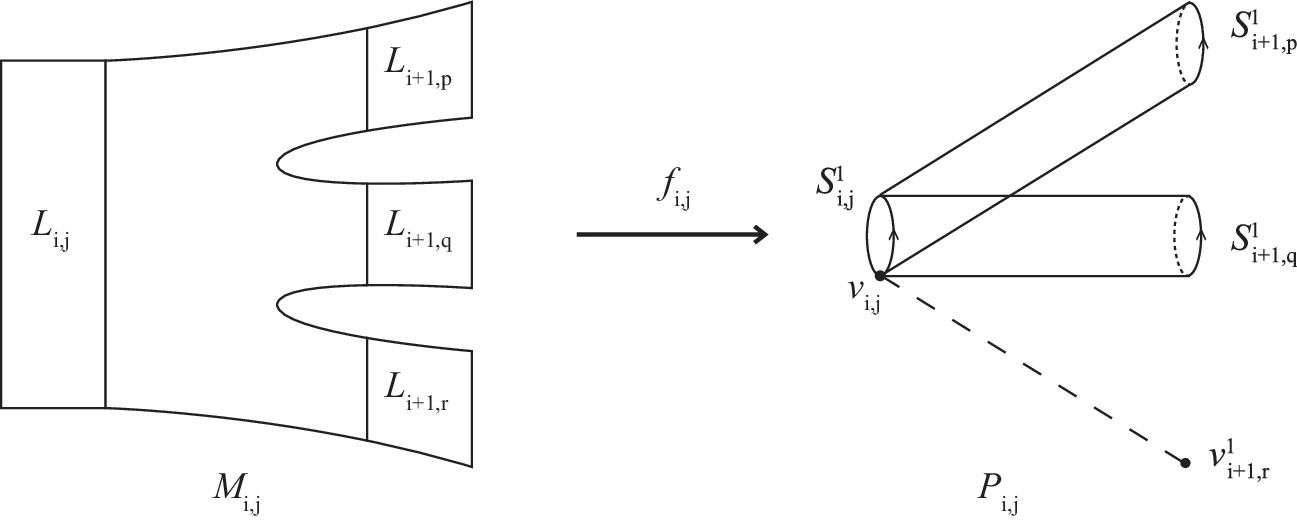}%
\caption{A building block of $f:J\backslash Y\rightarrow X_{\Gamma}$.}%
\label{Figure: Constructon of fij}%
\end{figure}
%EndExpansion

Choose a maximal tree $T_{i,j}$ in each $L_{i,j}$ then choose a maximal tree
$T_{i,j}^{^{\prime}}$ in each $M_{i,j}$ containing both $T_{i,j}$ and all
$T_{i+1,j^{\prime}}$ contained in $M_{i,j}$. Let $T=\cup T_{i,j}^{^{\prime}}$.
The tree-like structure of the collection $\left\{  M_{i,j}\right\}  $ ensures
that $T$ a maximal tree in $J\backslash Y$. Select a base vertex $p_{i,j}$
from each $L_{i,j}$, making sure that $p_{i,j}$ lies on the edge loop
$\alpha_{i,j}\subseteq L_{i,j}$ chosen previously. For each $L_{i+1,j^{\prime
}}$ on the right-hand side of an $M_{i,j}$, let $\lambda_{i+1,j^{\prime}}$ be
the unique edge path in $T_{i,j}^{\prime}$ from $p_{i,j}$ to $p_{i+1,j^{\prime
}}$.

Define $f:T\rightarrow X_{\Gamma}$ by sending each $T_{i,j}$ to $v_{i,j}$ and
every vertex of a $T_{i,j}^{\prime}$ not lying in one of those subtrees to
$v_{i,j}$. For each remaining edge $e$ of $T$, choose the $T_{i,j}^{\prime}$
containing it. If both ends of $e$ have been sent to $v_{i,j}$, send $e$ to
$v_{i,j}$; if one end has been sent to $v_{i,j}$ and the other to a
$v_{i+1,j^{\prime}}$, map $e$ homeomorphically onto $e_{i+1,j^{\prime}}$; if
one end lies in a $T_{i+1,j^{\prime}}$ and the other in a different
$T_{i+1,j^{\prime\prime}}$, send the midpoint of $e$ to $v_{i,j}$ and the two
halves of $e$ onto $e_{i+1,j^{\prime}}$ and $e_{i,j^{\prime\prime}}$, respectively.

Next we extend $f$ over the $L_{i,j}$. Each $L_{i,j}$ will be mapped into the
circle $S_{i,j}^{1}$, when that circle exists, otherwise to the vertex
$v_{i,j}$. Begin with $L_{0,1}=K_{0}$, which contains an oriented edge path
loop $\alpha_{0,1}$ that generates $H_{1}\left(  J\backslash Y\right)  $. Let
$\phi_{0,1}:H_{1}\left(  J\backslash Y\right)  \rightarrow\pi_{1}\left(
S_{0,1}^{1},v_{0,1}\right)  $ be the isomorphism taking $\alpha_{0,1}$ to the
positively oriented generator of $\pi_{1}\left(  S_{0,1},v_{0,1}\right)  $,
and consider the composition%
\[
\pi_{1}\left(  L_{0,1},p_{0,1}\right)  \twoheadrightarrow H_{1}\left(
L_{0,1}\right)  \twoheadrightarrow H_{1}\left(  J\backslash Y\right)
\overset{\phi_{0,1}}{\longrightarrow}\pi_{1}\left(  S_{0,1}^{1},v_{0,1}%
\right)  \text{.}%
\]
Recalling that $f$ has already been defined to send $T_{0,1}$ to $v_{0,1}$, we
extend over the remaining edges of $L_{0,1}$. If $e$ is one such edge then, by
giving it an orientation, it may be viewed as an element of $\pi_{1}\left(
L_{0,1},p_{0,1}\right)  $ and mapped into $S_{0,1}^{1}$ in accordance with its
image under the above homomorphism. Having mapped the 1-skeleton of $L_{0,1}$
into $S_{0,1}^{1}$ in accordance with a $\pi_{1}$-homomorphism, we may extend
to the 2-skeleton of $L_{0,1}$; then, by the asphericity of $S_{0,1}^{1}$, we
may extend to all of $L_{0,1}$. See, for example, \cite[\S 7.1]{Ge08}.

For general $L_{i,j}$, if $n_{i,j}=0$, send all of $L_{i,j}$ to $v_{i,j}$;
otherwise, the argument used above is repeated to map $L_{i,j}$ into
$S_{i,j}^{1}$, except that the map is based on the homomorphism%
\begin{equation}
\pi_{1}\left(  L_{i,j},p_{i,j}\right)  \twoheadrightarrow H_{1}\left(
L_{i,j}\right)  \twoheadrightarrow n_{i,j}\left\langle \alpha_{0,1}%
\right\rangle \overset{\phi_{i,j}}{\longrightarrow}\pi_{1}\left(  S_{i,j}%
^{1},v_{i,j}\right)  \label{homomorphism for defining f on L_ij}%
\end{equation}
where $n_{i,j}\left\langle \alpha_{0,1}\right\rangle \leq H_{1}\left(
J\backslash Y\right)  $ and $\phi_{i,j}$ is the (purely algebraic) isomorphism
taking the generator $n_{i,j}\alpha_{0,1}$ to the oriented generator of
$\pi_{1}\left(  S_{i,j}^{1},v_{i,j}\right)  $.

In the final step, we extend $f$ to all of $J\backslash Y$ by building maps
$f_{i,j}:M_{i,j}\rightarrow P_{i,j}$ that agree on their overlaps. In the
trivial cases, where $P_{i,j}$ is a wedge of of arcs, the existing map extends
to $M_{i,j}$ by the Tietze extension theorem. In the nontrivial cases,
$P_{i,j}$ strong deformation retracts onto $S_{i,j}^{1}$, and under that
retraction each $S_{i+1,j^{\prime}}^{1}$ is wrapped $k_{i+1,j^{\prime}}$ times
around $S_{i,j}^{1}$. Since $P_{i,j}$ is aspherical, we can use nearly the
same strategy as above, based on an analogous homomorphism
\[
\pi_{1}\left(  M_{i,j},p_{i,j}\right)  \twoheadrightarrow H_{1}\left(
M_{i,j}\right)  \twoheadrightarrow n_{i,j}\left\langle \alpha_{0,1}%
\right\rangle \overset{\psi_{i,j}}{\longrightarrow}\pi_{1}\left(
P_{i,j},v_{i,j}\right)
\]
On the subcomplex $T_{i,j}^{\prime}\cup L_{i,j}\cup\left(  \cup
L_{i+1,j^{\prime}}\right)  $ of $M_{i,j}$, where $f$ has already been defined,
the induced map into $\pi_{1}\left(  P_{i,j},v_{i,j}\right)  $ agrees with the
target homomorphism, so we may extend to the remaining edges, as dictated by
the homomorphism, and then to the remaining 2-cells, whose boundaries have
been sent to trivial loops in $P_{i,j}$. Finally, asphericity of $P_{i,j}$
allows us to inductively extend over the remaining cells of $M_{i,j}$.

\begin{proposition}
\label{Prop: proper 1-equivalence downstairs}The map $f:J\backslash
Y\rightarrow X_{\Gamma}$ is a proper 1-equivalence.
\end{proposition}

\begin{proof}
Since $\left\{  X_{\Gamma_{i}}\right\}  $ is a finite filtration of
$X_{\Gamma}$ and $f^{-1}\left(  X_{\Gamma_{i}}\right)  =K_{i}$ for each $i$,
$f$ is proper. To complete the proof, we construct a proper map $g^{\left(
1\right)  }:X_{\Gamma}^{\left(  1\right)  }\rightarrow J\backslash Y$ such
that $\left.  gf\right\vert _{J\backslash Y^{\left(  0\right)  }}$ is properly
homotopic to $J\backslash Y^{\left(  0\right)  }\hookrightarrow J\backslash Y$
and $\left.  fg\right\vert _{X^{\left(  0\right)  }}$ is properly homotopic to
$X_{\Gamma}^{\left(  0\right)  }\hookrightarrow X_{\Gamma}$.

For each $v_{i,j}\in X_{\Gamma}^{\left(  0\right)  }$, let $g\left(
v_{i,j}\right)  =p_{i,j}$. Map each $e_{i,j}$ originating at $v_{i,j}$ and
ending at $v_{i-1,j^{\prime}}$ homeomorphically onto the (reversed) edge path
$\lambda_{i,j}$ between $p_{i,j}$ and $p_{i-1,j^{\prime}}$ in
$T_{i-1,l^{j\prime}}^{\prime}$; and map each oriented $e_{i,j}^{\prime}$ once
around the oriented edge path loop $\alpha_{i,j}$ beginning and ending at
$p_{i,j}$.

Since $g^{\left(  1\right)  }\left(  P_{i,j}^{\left(  1\right)  }\right)
\subseteq M_{i,j}$, then $g^{\left(  1\right)  }\left(  Q_{i}^{\left(
1\right)  }\right)  \subseteq N_{i}$ for all $i$; so $g^{\left(  1\right)  }$
is proper. Notice that $\left.  fg^{\left(  1\right)  }\right\vert
_{X^{\left(  0\right)  }}=\operatorname*{id}_{X^{\left(  0\right)  }}$ and,
for each vertex $p\in J\backslash Y$, if $p\in M_{i,j}$, then $f\left(
p\right)  \in P_{i,j}$; so $g^{\left(  1\right)  }\left(  f\left(  p\right)
\right)  \in M_{i,j}$. A choice of edge path $\mu_{p}$ in $M_{i,j}$ from $p$
to $g^{\left(  1\right)  }\left(  f\left(  p\right)  \right)  $ for each $p\in
M_{i,j}$ determines a proper homotopy between the inclusion and $\left.
g^{\left(  1\right)  }f\right\vert _{J\backslash Y^{\left(  0\right)  }}$.
\end{proof}

\begin{remark}
\label{Remark: proper 1-equivalence plus more}The above construction
accomplishes more than required for a 1-equivalence; specifically,
$fg^{\left(  1\right)  }$ is properly homotopic to $X_{\Gamma}^{\left(
1\right)  }\hookrightarrow X_{\Gamma}$. To see this, note that each oriented
edge $e_{i,j}$ from $v_{i,j}$ to $v_{i-1,j^{\prime}}$ is mapped by $g^{\left(
1\right)  }$ to the edge path $\lambda_{i,j}$ from $p_{i,j}$ to
$p_{i-1,j^{\prime}}$ and $f$ sends $\lambda_{i,j}$ entirely into $e_{i,j}$
with $f\left(  p_{i,j}\right)  =v_{i,j}$ and $f\left(  p_{i-1,j^{\prime}%
}\right)  =v_{i-1,j^{\prime}}$. A discrete collection of straightening
homotopies, each supported in an edge $e_{i,j}$ and fixing all vertices,
combine to properly homotope $fg^{\left(  1\right)  }$ to the identity over
the tree $\Gamma$. For the \textquotedblleft loop edges\textquotedblright%
\ $e_{i,j}^{\prime}$, the story is similar. The map $g^{\left(  1\right)  }$
takes $e_{i,j}^{\prime}$ once around $\alpha_{i,j}\subseteq L_{i,j}$ and $f$
returns $\alpha_{i,j}$ (in fact, all of $L_{i,j}$) to $S_{i,j}^{1}%
=e_{i,j}^{\prime}\cup v_{i,j}$, with vertices going to $v_{i,j}$, some edges
sent entirely to $v_{i,j}$ and others around $S_{i,j}^{1}$ (possibly multiple
times, in the forward or reverse directions). Since homomorphism
\ref{homomorphism for defining f on L_ij}, used to define $f$ on $L_{i,j}$,
takes $\alpha_{i,j}$ to the positively oriented generator of $\pi_{1}\left(
S_{i,j}^{1},v_{i,j}\right)  $, $\left.  fg^{\left(  1\right)  }\right\vert
_{e_{i,j}^{\prime}}$ is homotopic to the identity by a base point preserving
homotopy supported in $e_{i,j}$. A discrete collection of such homotopies
completes the straightening process.\medskip
\end{remark}

The proper 1-inverse $g^{\left(  1\right)  }:X_{\Gamma}^{\left(  1\right)
}\rightarrow J\backslash Y$ of $f:J\backslash Y\rightarrow X_{\Gamma}$ becomes
more useful when extended to all of $X_{\Gamma}$, even if that extension is
not proper. With the aid of a \textquotedblleft strongly
coaxial\textquotedblright\ hypothesis, a proper extension becomes possible.

\begin{proposition}
\label{Prop: proper 2-equivalence downstairs}The map $g^{\left(  1\right)
}:X_{\Gamma}^{\left(  1\right)  }\rightarrow J\backslash Y$ constructed in the
proof of Proposition \ref{Prop: proper 1-equivalence downstairs} can always be
extended to a map $g:X_{\Gamma}\rightarrow J\backslash Y$ that induces a
$\pi_{1}$-isomorphism. If $j$ is strongly coaxial, $g$ can be chosen to be a
proper 2-inverse for $f$.
\end{proposition}

\begin{proof}
To obtain $g:X_{\Gamma}\rightarrow J\backslash Y$, we need only extend
$g^{\left(  1\right)  }$ over the 2-cells $d_{i,j}$ of $X_{\Gamma}$. Each
$d_{i,j}$ is glued to $X_{\Gamma}^{\left(  1\right)  }$ along a loop of the
form $(e_{i,j}^{\prime})^{k_{i,j}}\cdot e_{i,j}\cdot(e_{i-1,j^{\prime}%
}^{\prime})^{-1}\cdot(e_{i,j})^{-1}$, and that loop is mapped to
$(\alpha_{i,j})^{k_{i,j}}\cdot\lambda_{i,j}\cdot(\alpha_{i-1,j^{\prime}}%
)^{-1}\cdot(\lambda_{i,j})^{-1}$ which is homologically, and hence
homotopically, trivial in $J\backslash Y$. So the map can be extended.

If $j$ is strongly coaxial, then by Lemma \ref{Lemma: Basic Lemma 1}, we may
(by passing to a subsequence of $\left\{  K_{i}\right\}  $) assume that, for
each $n\geq1$, loops in $J\backslash Y-K_{n}$ that are null-homotopic in
$J\backslash Y$ contract in $J\backslash Y-K_{n-1}$.\footnote{Actually,
passing to a subsequence changes the corresponding model tree $\Gamma$, and
thus $X_{\Gamma}$. That change is precisely a reduction of $\Gamma$ to a
$\Gamma^{\prime}$, as discussed in
\S \ref{Subsection: Reductions of model spaces}. By Proposition
\ref{Prop: Result of a reduction}, that change does not affect the proper
homotopy type of $X_{\Gamma}$.} Since the attaching loop for each $d_{i,j}$,
lies in $P_{i}\subseteq X_{\Gamma}-X_{\Gamma_{i-1}}$, its image $(\alpha
_{i,j})^{k_{i,j}}\cdot\lambda_{i,j}\cdot(\alpha_{i-1,j^{\prime}})^{-1}%
\cdot(\lambda_{i,j})^{-1}$ lies in $N_{i}\subseteq J\backslash Y-K_{i-1}$ and
is homotopically trivial in $J\backslash Y$. Therefore it contracts in
$J\backslash Y-K_{i-2}$. Use these contractions to extend $g^{\left(
1\right)  }$ over the 2-cells of $X_{\Gamma}$ to obtain a proper map
$g:X_{\Gamma}\rightarrow J\backslash Y$. In light of Remark
\ref{Remark: proper 1-equivalence plus more}, it remains only to show that
$\left.  gf\right\vert _{J\backslash Y^{\left(  1\right)  }}$ is properly
homotopic to $J\backslash Y^{\left(  1\right)  }\hookrightarrow J\backslash
Y$. First we obtain the desired homotopy on the maximal tree $T\subseteq
J\backslash Y^{\left(  1\right)  }$ used in defining $f$. In the proof of
Proposition \ref{Prop: proper 1-equivalence downstairs} we obtained a proper
homotopy between $J\backslash Y^{\left(  0\right)  }\hookrightarrow
J\backslash Y$ and $\left.  gf\right\vert _{J\backslash Y^{\left(  0\right)
}}$ by choosing a proper family of edge paths $\mu_{p}$ between $p$ and
$gf\left(  p\right)  $ for $p\in Y^{\left(  0\right)  }$. Moving inductively
outward from the base vertex $p_{0}$, we can rechoose the $\mu_{p}$, if
necessary, so the loops $e\cdot\mu_{q}\cdot(gf(e))^{-1}\cdot\mu_{p}^{-1}$,
where $e$ is an edge in $T$ from $p$ to $q$, bound a proper collection of
singular disks in $Y$; in particular, if $e$ lies in $T_{i,j}^{\prime}$,
arrange for the disk to lie in $M_{i,j}$. Together these disks determine a
proper homotopy on $T$. To complete the homotopy, let $e$ be an edge in
$J\backslash Y^{\left(  1\right)  }-T$. Choose the $M_{i,j}$ containing $e$
and let $\lambda_{p}$ and $\lambda_{q}$ be reduced edge paths in
$T_{i,j}^{\prime}$ connecting $p_{i,j}$ to the initial and terminal points $p$
and $q$ of $e$, respectively. By construction of $f$ and $g^{\left(  1\right)
}$,$\lambda_{p}\cdot e\cdot\lambda_{q}^{-1}$ and $g^{\left(  1\right)
}f\left(  \lambda_{p}\cdot e\cdot\lambda_{q}^{-1}\right)  $ are homotopic in
$J\backslash Y$, by choice of the $K_{i}$, they are homotopic in $J\backslash
Y-K_{i-2}$. Since a homotopy has already been constructed between these loops
away from $e$, with tracks $\mu_{p}$ and $\mu_{q}$ at $p$ and $q$,
respectively, it must be that $e\cdot\mu_{q}\cdot(gf(e))^{-1}\cdot\mu_{p}%
^{-1}$ is null homotopic in $J\backslash Y-K_{i-2}$. Filling each such loop
with a singular disk completes the proper homotopy between $J\backslash
Y^{\left(  1\right)  }\hookrightarrow J\backslash Y$ and $\left.
gf\right\vert _{J\backslash Y^{\left(  1\right)  }}$.
\end{proof}

\begin{corollary}
\label{Cor: lifting proper 1- and 2-equivalences}Let $Y$ be a simply
connected, strongly locally finite CW complex, and $j:Y\rightarrow Y$ a
cellular homeomorphism generating an action by covering transformations with
$J\equiv\left\langle j\right\rangle \cong%
%TCIMACRO{\U{2124} }%
%BeginExpansion
\mathbb{Z}
%EndExpansion
$, and let $\Gamma$ be the corresponding model tree. Then $Y$ is $%
%TCIMACRO{\U{2124} }%
%BeginExpansion
\mathbb{Z}
%EndExpansion
$-equivariantly proper 1-equivalent to the model $%
%TCIMACRO{\U{2124} }%
%BeginExpansion
\mathbb{Z}
%EndExpansion
$-space $\widetilde{X}_{\Gamma}$. If $j$ is strongly coaxial, then $Y$ is $%
%TCIMACRO{\U{2124} }%
%BeginExpansion
\mathbb{Z}
%EndExpansion
$-equivariantly proper 2-equivalent to $\widetilde{X}_{\Gamma}$.
\end{corollary}

\begin{proof}
In the general case, the proper 1-equivalence $f:$ $J\backslash Y\rightarrow
X_{\Gamma}$ lifts to a $%
%TCIMACRO{\U{2124} }%
%BeginExpansion
\mathbb{Z}
%EndExpansion
$-equivariant proper 1-equivalence $\widetilde{f}:$ $Y\rightarrow
\widetilde{X}_{\Gamma}$ whose proper equivariant 1-inverse is obtained by
lifting the (not necessarily proper) $g:X_{\Gamma}\rightarrow J\backslash Y$
to $\widetilde{g}:\widetilde{X}_{\Gamma}\rightarrow Y$, then noting that its
restriction $\widetilde{g}^{\left(  1\right)  }$to $\widetilde{X}_{\Gamma
}^{\left(  1\right)  }$ (the 1-skeleton of $\widetilde{X}_{\Gamma}$, not the
universal cover of $X_{\Gamma}^{\left(  1\right)  }$), being a lift of
$g^{\left(  1\right)  }:X_{\Gamma}^{\left(  1\right)  }\rightarrow J\backslash
Y$, is proper.

When $j$ is strongly coaxial, the proper 2-equivalences $f:$ $J\backslash
Y\rightarrow X_{\Gamma}$ and $g:X_{\Gamma}\rightarrow J\backslash Y$ lift to $%
%TCIMACRO{\U{2124} }%
%BeginExpansion
\mathbb{Z}
%EndExpansion
$-equivariant proper 2-equivalences $\widetilde{f}:$ $Y\rightarrow
\widetilde{X}_{\Gamma}$ and $\widetilde{g}:\widetilde{X}_{\Gamma}\rightarrow
Y$.
\end{proof}

We now address the situation where $j$ is only assumed to be coaxial. With
significant additional effort, we will recover nearly the full strength of
Corollary \ref{Cor: lifting proper 1- and 2-equivalences}.

\begin{proposition}
\label{Prop: proper 2-equiv with only coaxial hypothesis}Let $Y$ be a strongly
locally finite CW complex and $j:Y\rightarrow Y$ a cellular homeomorphism
generating a proper rigid action with $J\equiv\left\langle j\right\rangle
\cong%
%TCIMACRO{\U{2124} }%
%BeginExpansion
\mathbb{Z}
%EndExpansion
$, and let $\Gamma$ be a corresponding model tree. If $j$ is coaxial, then $Y$
is proper 2-equivalent to $\widetilde{X}_{\Gamma}$ via maps that are $%
%TCIMACRO{\U{2124} }%
%BeginExpansion
\mathbb{Z}
%EndExpansion
$-equivariant on the 1-skeleta of $Y$ and $\widetilde{X}_{\Gamma}$.
\end{proposition}

Our starting point for the proof of Proposition
\ref{Prop: proper 2-equiv with only coaxial hypothesis} is the already
existing diagram
\begin{equation}%
\begin{tabular}
[c]{ccccc}%
$\ Y$ & $\overset{\widetilde{f}}{\longrightarrow}$ & $\ \widetilde{X}_{\Gamma
}$ & $\overset{\widetilde{\rho}}{\longrightarrow}$ & $\
%TCIMACRO{\U{211d} }%
%BeginExpansion
\mathbb{R}
%EndExpansion
$\\
$p\downarrow$ &  & $q\downarrow$ &  & $r\downarrow$\\
$\ J\backslash Y$ & $\overset{f}{\longrightarrow}$ & $\ X_{\Gamma}$ &
$\overset{\rho}{\longrightarrow}$ & $\ S_{0,1}^{1}$%
\end{tabular}
\ \ \ \ \ \ \ \
\end{equation}
Discard the lift $\widetilde{g}:\widetilde{X}_{\Gamma}\rightarrow Y$, since it
may not be proper under the new hypothesis; but retain its restriction
$\widetilde{g}^{\left(  1\right)  }:\widetilde{X}_{\Gamma}^{\left(  1\right)
}\rightarrow Y$, which is a proper 1-inverse for $\widetilde{f}$. We will
construct an alternative extension $\overline{g}:\widetilde{X}_{\Gamma
}\rightarrow Y$ of $\widetilde{g}^{\left(  1\right)  }$ which is a proper
2-inverse for $\widetilde{f}$. By lifting the homotopy noted in Remark
\ref{Remark: proper 1-equivalence plus more}, we already have an equivariant
proper homotopy between\emph{ }$\widetilde{X}_{\Gamma}^{\left(  1\right)
}\hookrightarrow\widetilde{X}_{\Gamma}$ and $\widetilde{f}\widetilde{g}%
^{\left(  1\right)  }$; so it is enough to obtain a proper extension
$\overline{g}:\widetilde{X}_{\Gamma}\rightarrow Y$ and to show that $\left.
\overline{g}\widetilde{f}\right\vert _{Y^{\left(  1\right)  }}=\left.
\widetilde{g}^{\left(  1\right)  }\widetilde{f}\right\vert _{Y^{\left(
1\right)  }}$ is properly homotopic to $Y^{\left(  1\right)  }\hookrightarrow
Y$. Both tasks depend upon the coaxial hypothesis.

Before launching into the proof, we introduce some notation and prove a few
easy lemmas.

\begin{itemize}
\item For $\left[  r,s\right]  \subseteq%
%TCIMACRO{\U{211d} }%
%BeginExpansion
\mathbb{R}
%EndExpansion
$, let $\widetilde{X}_{\Gamma}^{\left[  r,s\right]  }=\widetilde{\rho}%
^{-1}\left(  \left[  r,s\right]  \right)  $ and
\[
Y^{\left[  r,s\right]  }=\widetilde{f}^{-1}\left(  \widetilde{X}_{\Gamma
}^{\left[  r,s\right]  }\right)  =\left(  \widetilde{\rho}\widetilde{f}%
\right)  ^{-1}\left(  \left[  r,s\right]  \right)  .
\]
More generally, if $P\subseteq Y$, then $P^{\left[  r,s\right]  }=P\cap
Y^{\left[  r,s\right]  }$, and if $Q\subseteq$ $\widetilde{X}_{\Gamma}$, then
$Q^{\left[  r,s\right]  }=Q\cap\widetilde{X}_{\Gamma}^{\left[  r,s\right]  }$.
We will use similar notation for arbitrary $S\subseteq%
%TCIMACRO{\U{211d} }%
%BeginExpansion
\mathbb{R}
%EndExpansion
$, such as $Y^{S}$ or $P^{S}$.

\item A \emph{level set} in $Y$ is a set of the form $Y^{\left\{  r\right\}
}$ or $P^{\left\{  r\right\}  }$, for $r\in%
%TCIMACRO{\U{211d} }%
%BeginExpansion
\mathbb{R}
%EndExpansion
$; level sets in $\widetilde{X}_{\Gamma}$ are defined similarly.

\item The \emph{height} of $P\subseteq Y$ is the diameter of $\widetilde{\rho
}\widetilde{f}\left(  P\right)  $ in $%
%TCIMACRO{\U{211d} }%
%BeginExpansion
\mathbb{R}
%EndExpansion
$; the height of $Q\subseteq$ $\widetilde{X}_{\Gamma}$ is the diameter of
$\widetilde{\rho}\left(  Q\right)  $.
\end{itemize}

Let $\left\{  K_{i}\right\}  $ be a nested exhaustion of $J\backslash Y$ by
finite connected complexes satisfying all of the basic conditions used in
constructing the model spaces, and recall the associated sequence of
neighborhoods of infinity $\left\{  N_{i}\right\}  $ and the finite
subcomplexes $L_{i}=N_{i}\cap K_{i}$ and $M_{i}=N_{i}\cap K_{i+1}$. Notice
that $\left\{  \widetilde{K}_{i}^{\left[  -i,i\right]  }\right\}
_{i=1}^{\infty}$ is a nested exhaustion of $Y$ by finite subcomplexes. By
applying the coaxial hypothesis inductively, we may (by passing to a
subsequence, then relabeling) assume that, for all $i$, loops in
$Y-\widetilde{K}_{i+1}$ contract in $Y-\widetilde{K}_{i}^{\left[  -i,i\right]
}$. For convenience, let $\Gamma$, $X_{\Gamma}$, and $\widetilde{X}_{\Gamma}$
be the models based on that exhaustion of $J\backslash Y$, and let
$f:J\backslash Y\rightarrow X_{\Gamma}$ be a corresponding map. (By
Proposition \ref{Prop: Result of a reduction}, this does not affect the proper
homotopy type of $X_{\Gamma}$ or the equivariant proper homotopy type of
$\widetilde{X}_{\Gamma}$.) Then, for the canonical finite exhaustion,
$\left\{  X_{\Gamma_{i}}\right\}  _{i=0}^{\infty}$ of $X_{\Gamma}$, the
corresponding neighborhoods of infinity $X_{\Gamma}=Q_{0}\supseteq
Q_{1}\supseteq Q_{2}\supseteq\cdots$, where $Q_{i}=\overline{X_{\Gamma
}-X_{\Gamma_{i}}}$, and the subcomplexes $P_{i}=Q_{i}\cap X_{\Gamma_{i+1}}$,
the following is immediate from the construction of $f$.

\begin{lemma}
\label{Lemma 2}Given the above setup, $\left\{  \widetilde{K}_{i}^{\left[
-i,i\right]  }\right\}  $ is a finite exhaustion of $Y$; $\left\{
\widetilde{X}_{\Gamma_{i}}^{\left[  -i,i\right]  }\right\}  $ is a finite
exhaustion of $\widetilde{X}_{\Gamma}$; and $\widetilde{f}:Y\rightarrow
\widetilde{X}_{\Gamma}$ is level-preserving and satisfies the following
properties for all $i$.

\begin{enumerate}
\item $\widetilde{f}\left(  \widetilde{K}_{i}\right)  =\widetilde{X}%
_{\Gamma_{i}}$,

\item $\widetilde{f}\left(  Y-\widetilde{K}_{i}\right)  =\widetilde{X}%
_{\Gamma}-\widetilde{X}_{\Gamma_{i}}$,

\item $\widetilde{f}\left(  \widetilde{M}_{i}\right)  =\widetilde{P}_{i}$, and

\item $\widetilde{f}\left(  \widetilde{K}_{i}^{S}\right)  =\widetilde{X}%
_{\Gamma_{i}}^{S}$ for all $S\subseteq%
%TCIMACRO{\U{211d} }%
%BeginExpansion
\mathbb{R}
%EndExpansion
$.\smallskip
\end{enumerate}
\end{lemma}

The construction of $g^{\left(  1\right)  }:X_{\Gamma}^{\left(  1\right)
}\rightarrow Y$ leads to similar properties for its lift.

\begin{lemma}
\label{Lemma 2.5}The function $\widetilde{g}^{\left(  1\right)  }%
:\widetilde{X}_{\Gamma}^{(1)}\rightarrow Y$ is level-preserving and satisfies
the following properties for all $i$.

\begin{enumerate}
\item $\widetilde{g}^{\left(  1\right)  }\left(  \widetilde{X}_{\Gamma_{i}%
}^{\left(  1\right)  }\right)  \subseteq\widetilde{K}_{i}^{\left(  1\right)
}$,

\item $\widetilde{g}^{\left(  1\right)  }\left(  \widetilde{Q}_{i}^{\left(
1\right)  }\right)  \subseteq\widetilde{N}_{i}^{\left(  1\right)  }$,

\item $\widetilde{g}^{\left(  1\right)  }\left(  \widetilde{P}_{i}^{\left(
1\right)  }\right)  \subseteq\widetilde{M}_{i}^{\left(  1\right)  }$, and

\item $\widetilde{g}^{\left(  1\right)  }\left(  \left(  \widetilde{X}%
_{\Gamma_{i}}^{\left(  1\right)  }\right)  ^{S}\right)  \subseteq\left(
\widetilde{K}_{i}^{\left(  1\right)  }\right)  ^{S}$ for all $S\subseteq%
%TCIMACRO{\U{211d} }%
%BeginExpansion
\mathbb{R}
%EndExpansion
$.\smallskip
\end{enumerate}
\end{lemma}

The following refinement of items 2) in Lemmas \ref{Lemma 2} and
\ref{Lemma 2.5} says that $\widetilde{f}$ and $\widetilde{g}^{\left(
1\right)  }$ also respect the components of $\left\{  \widetilde{N}%
_{i}\right\}  $ and $\left\{  \widetilde{Q}_{i}\right\}  $.

\begin{lemma}
\label{Lemma: correspondence of filtered components}Let $i$ be fixed and
$\left\{  E_{k}\right\}  _{k=1}^{i_{0}}$ the finite collection of path
components of $\widetilde{Q}_{i}$. Then $\widetilde{N}_{i}$ has an equal
number of components and $\widetilde{f}$ induces a bijection between those
collections. If we label the components of $\widetilde{N}_{i}$ by $\left\{
F_{k}\right\}  _{k=1}^{i_{0}}$ so that $\widetilde{f}\left(  F_{k}\right)
=E_{k}$ for each $k$, then $\widetilde{g}^{\left(  1\right)  }$ takes
$E_{k}^{\left(  1\right)  }$ into $F_{k}^{\left(  1\right)  }$.
\end{lemma}

\begin{remark}
\label{Remark: on correspondence of filterred components}A similar
correspondence between components of $\widetilde{M}_{i}$ and
$\widetilde{\overline{X_{\Gamma_{i}}-X_{\Gamma_{i-1}}}}$ can be
deduced.\medskip
\end{remark}

\begin{lemma}
\label{Lemma 4}For each $i$, there is an integer $p_{i}$ such that any two
points in a level set $\widetilde{K}_{i}^{\left\{  r\right\}  }$ can be
connected by a path in $\widetilde{K}_{i}$ of height $\leq p_{i}$.
\end{lemma}

\begin{proof}
$\widetilde{K_{i}}$ is a connected complex and $\widetilde{K}_{i}^{\left[
0,1\right]  }$ is compact, so there exists an interval $\left[  -k,k\right]  $
so that points in $\widetilde{K}_{i}^{\left[  0,1\right]  }$ can be connected
in $\widetilde{K}_{i}^{\left[  -k,k\right]  }$. Since $\widetilde{K}%
_{i}^{\left\{  r\right\}  }$ is $J$-equivalent to a level set lying in
$\widetilde{K}_{i}^{\left[  0,1\right]  }$ we can let $p_{i}=2k$.
\end{proof}

By essentially the same argument we have:

\begin{lemma}
\label{Lemma: small paths in shells}For each $i$, there is an integer $q_{i}$
such that any two points in a level set $\widetilde{M}_{i}^{\left\{
r\right\}  }$ that lie in the same component of $\widetilde{M}_{i}$ can be
connected by a path, in that component, of height $\leq q_{i}$.
\end{lemma}

\begin{lemma}
\label{Lemma 5}For each triple $(i,h,r)\in%
%TCIMACRO{\U{2115} }%
%BeginExpansion
\mathbb{N}
%EndExpansion
^{3}$, there exists $s(i,h,r)\in%
%TCIMACRO{\U{2115} }%
%BeginExpansion
\mathbb{N}
%EndExpansion
$ such that loops in $\widetilde{K}_{i}^{(-\infty,-s]\cup\lbrack s,\infty)}$
of height $\leq h$ contract in $Y^{(-\infty,-r]\cup\lbrack r,\infty)}.$
\end{lemma}

\begin{proof}
Since $\widetilde{K}_{i}^{\left[  0,2h\right]  }$ is compact and $Y$ is simply
connected, there exists an integer $t>0$ so that all loops in $\widetilde{K}%
_{i}^{\left[  0,2h\right]  }$ contract in $Y^{\left[  -t,2h+t\right]  }$. So
by $J$-translation, for every integer $k$, loops lying in $\widetilde{K}%
_{i}^{\left[  k,k+2h\right]  }$ contract in $Y^{[k-t,\infty)}$. Let $s=r+t+1$
and note that every loop in $\widetilde{K}_{i}^{[s,\infty)}$ of height $\leq
h$ lies in $\widetilde{K}_{i}^{\left[  k,k+2h\right]  }$ for some integer
$k\geq r+t$.

A similar calculation handles loops of height $\leq h$ lying in $\widetilde{K}%
_{i}^{(-\infty,-s]}$.
\end{proof}

\begin{lemma}
\label{Lemma: bounded height}For each $i\in%
%TCIMACRO{\U{2115} }%
%BeginExpansion
\mathbb{N}
%EndExpansion
$, there exists $h_{i}\in%
%TCIMACRO{\U{2115} }%
%BeginExpansion
\mathbb{N}
%EndExpansion
$ such that the 2-cells of $\widetilde{X}_{\Gamma_{i}}$ have height $\leq
h_{i}$.
\end{lemma}

\begin{proof}
The $2$-cells of $\widetilde{X}_{\Gamma_{1}}$ that lie over a 2-cell $d_{1,j}$
of $X_{\Gamma_{1}}$ have height $k_{1,j}$. Moving outward, $2$-cells of
$\widetilde{X}_{\Gamma_{2}}$ that lie over a $d_{2,j}$ have height
$k_{2,j}\cdot k_{1,j^{\prime}}$, where $v_{1,j^{\prime}}$ is the terminal
vertex of $e_{2,j}$. In general, the height of a 2-cell of $\widetilde{X}%
_{\Gamma}$ lying over a 2-cell $d_{i,j}$ in $X_{\Gamma}$ is equal to the
product of the labels on the edge path connecting $v_{i,j}$ to $v_{0,1}$. So
heights of the 2-cells in $\widetilde{X}_{\Gamma_{i}}$ are bounded by the
largest such product.\bigskip
\end{proof}

\begin{remark}
In contrast to the increasing heights of the 2-cells of $\widetilde{X}%
_{\Gamma}$ as their distances from the central axis $v_{0,1}\times%
%TCIMACRO{\U{211d} }%
%BeginExpansion
\mathbb{R}
%EndExpansion
$ increases, the widths of the 2-cells are constantly 1, when viewed as
subsets of $\Gamma^{+}\times%
%TCIMACRO{\U{211d} }%
%BeginExpansion
\mathbb{R}
%EndExpansion
\subseteq\widetilde{X}_{\Gamma}$ and measured in the $\Gamma^{+}$-direction.
In the argument that follows, we refer to this property as the
\textquotedblleft narrowness of the 2-cells of $\widetilde{X}_{\Gamma}%
$\textquotedblright.\bigskip
\end{remark}

\begin{proof}
[Completion of the proof of Prop.
\ref{Prop: proper 2-equiv with only coaxial hypothesis}]We will construct a
proper 2-inverse $\overline{g}:\widetilde{X}_{\Gamma}\rightarrow Y$ for
$\widetilde{f}:Y\rightarrow\widetilde{X}_{\Gamma}$, by extending
$\widetilde{g}^{\left(  1\right)  }:\widetilde{X}_{\Gamma}^{(1)}\rightarrow Y$
over the 2-cells of $\widetilde{X}_{\Gamma}$. The fact that $\overline{g}$ is
$J$-equivariant and level-preserving on $\widetilde{X}_{\Gamma}^{\left(
1\right)  }$ is immediate. To assure properness of $\overline{g}$, we will
arrange that, for each $i$, only finitely many 2-cells have images
intersecting $\widetilde{K}_{i}^{\left[  -i,i\right]  }$. A similar strategy
will give the required proper homotopies. Both constructions rely on the
coaxial hypothesis.\smallskip

\noindent\textbf{Claim.} \emph{For each }$i\in%
%TCIMACRO{\U{2115} }%
%BeginExpansion
\mathbb{N}
%EndExpansion
$\emph{, there exists }$s_{i}\in%
%TCIMACRO{\U{2115} }%
%BeginExpansion
\mathbb{N}
%EndExpansion
$\emph{ such that, if }$\sigma$\emph{ is a }$2$\emph{-cell of }$\widetilde{X}%
_{\Gamma}$\emph{ lying outside }$\widetilde{X}_{\Gamma_{i+1}}^{\left[
-s_{i},s_{i}\right]  }$\emph{, then }$\left.  \widetilde{g}^{\left(  1\right)
}\right\vert _{\partial\sigma}$\emph{ extends to a map of }$\sigma$\emph{ into
}$Y-\widetilde{K}_{i}^{\left[  -i,i\right]  }\emph{.\smallskip}$

Let $h_{i+2}$ be the integer supplied by Lemma \ref{Lemma: bounded height};
then let $s_{i}=s(i+2,h_{i+2},i+1)$, as promised in Lemma \ref{Lemma 5}%
.\smallskip

\noindent\textsc{Case 1.} $\sigma\subseteq\widetilde{X}_{\Gamma}%
-\widetilde{X}_{\Gamma_{i+1}}$.\smallskip

By Lemma \ref{Lemma 2.5}, $\widetilde{g}^{\left(  1\right)  }$ takes
$\partial\sigma$ into $Y-\widetilde{K}_{i+1}$, so by hypothesis and choice of
$\left\{  K_{i}\right\}  $, $\left.  \widetilde{g}^{\left(  1\right)
}\right\vert _{\partial\sigma}$ extends to a map taking $\sigma$ into
$Y-\widetilde{K}_{i}^{\left[  -i,i\right]  }$.\smallskip

\noindent\textsc{Case 2.} $\sigma$\emph{ is not contained in }$\widetilde{X}%
_{\Gamma}-\widetilde{X}_{\Gamma_{i+1}}$\emph{.}\smallskip

By narrowness of 2-cells in $\widetilde{X}_{\Gamma}$, $\sigma$ lies in
$\widetilde{X}_{\Gamma_{i+2}}$; and since $\sigma$ lies outside $\widetilde{X}%
_{\Gamma_{i+1}}^{\left[  -s_{i},s_{i}\right]  }$, it lies in $\widetilde{X}%
_{\Gamma_{i+2}}^{(-\infty,-s_{i}]\cup\lbrack s_{i},\infty)}$. By Lemma
\ref{Lemma: bounded height}, $\sigma$ has height $\leq h_{i+2}$, so by Lemma
\ref{Lemma 2.5}, $\widetilde{g}^{\left(  1\right)  }$ takes $\partial\sigma$
to a loop in $\widetilde{K}_{i+2}^{(-\infty,-s_{i}]\cup\lbrack s_{i},\infty)}$
of height $\leq h_{i+2}$. By choice of $s_{i}$, $\left.  \widetilde{g}%
^{\left(  1\right)  }\right\vert _{\partial\sigma}$ extends to a map of
$\sigma$ into $Y^{(-\infty,-(i+1)]\cup\lbrack i+1,\infty)}\subseteq
Y-\widetilde{K}_{i}^{\left[  -i,i\right]  }$.\smallskip

With the claim proved, we define $\overline{g}$ inductively, as follows. Let
$\left(  s_{i}\right)  _{i\in%
%TCIMACRO{\U{2115} }%
%BeginExpansion
\mathbb{N}
%EndExpansion
}$ be a strictly increasing sequence of integers satisfying the claim. To get
started, use simple-connectivity of $Y$ to extend $\widetilde{g}^{\left(
1\right)  }$ over all of the (finitely many) 2-cells of $\widetilde{X}%
_{\Gamma}$ that intersect $\widetilde{X}_{\Gamma_{2}}^{\left[  -s_{1}%
,s_{1}\right]  }$. Then extend over the 2-cells that miss $\widetilde{X}%
_{\Gamma_{2}}^{\left[  -s_{1},s_{1}\right]  }$ but intersect $\widetilde{X}%
_{\Gamma_{3}}^{\left[  -s_{2},s_{2}\right]  }$, using the choice of $s_{1}$ to
ensure their images miss $\widetilde{K}_{1}^{\left[  -1,1\right]  }$. Next,
extend over the 2-cells that miss $\widetilde{X}_{\Gamma_{3}}^{\left[
-s_{2},s_{2}\right]  }$ but intersect $\widetilde{X}_{\Gamma_{4}}^{\left[
-s_{3},s_{3}\right]  }$, making sure that their images miss $\widetilde{K}%
_{2}^{\left[  -2,2\right]  }$. Continue inductively to obtain a proper map
$\overline{g}:\widetilde{X}_{\Gamma}\rightarrow Y$.\medskip

To conclude that $\overline{g}$ is a proper 2-inverse for $\widetilde{f}$, we
must show that the restrictions of $\overline{g}\widetilde{f}$ and
$\widetilde{f}\overline{g}$ to the 1-skeleta of their respective domains are
properly homotopic to inclusion maps. The second of these requires no work;
just lift the proper homotopy described in Remark
\ref{Remark: proper 1-equivalence plus more}. It remains to construct a proper
homotopy between $Y^{\left(  1\right)  }\hookrightarrow Y$ and $\left.
\overline{g}\widetilde{f}\right\vert _{Y^{\left(  1\right)  }}$.

We first construct the homotopy over the $0$-skeleton of $Y$. Let $v$ be a
vertex of $Y$ and $v^{\prime}=$ $\overline{g}\widetilde{f}\left(  v\right)  $.
Choose an integer $i$ so that $v\in\widetilde{M}_{i}$. By Lemma
\ref{Lemma: correspondence of filtered components} and Remark
\ref{Remark: on correspondence of filterred components}, $v$ and $v^{\prime}$
lie in the same component of $\widetilde{M}_{i}$ and, since $\left.
\overline{g}\widetilde{f}\right\vert _{Y^{\left(  0\right)  }}$ is
level-preserving, Lemma \ref{Lemma: small paths in shells} guarantees a path
$\alpha_{v}$ from $v$ to $v^{\prime}$ in that component with height $\leq
q_{i}$. By parameterizing each $\alpha_{v}$ over $\left[  0,1\right]  $, we
obtain a proper homotopy $H_{t}^{\left(  0\right)  }$ between $Y^{\left(
0\right)  }\hookrightarrow Y$ and $\left.  \overline{g}\widetilde{f}%
\right\vert _{Y^{\left(  0\right)  }}$.

To extend $H_{t}^{\left(  0\right)  }$ over the edges of $Y^{\left(  1\right)
}$, let $e$ be a fixed (oriented) edge between vertices $v_{1}$ and $v_{2}$ in
$Y$. Since $e$ is a lift of an edge from $J\backslash Y$, $e$ lies in a
component of some $\widetilde{M}_{i}$. By Lemma
\ref{Lemma: correspondence of filtered components} and Remark
\ref{Remark: on correspondence of filterred components}, the oriented path
$e^{\prime}=$ $\widetilde{g}\widetilde{f}\left(  e\right)  $ lies in the same
component. Let $\beta_{e}$ be the loop $e\ast\alpha_{v_{2}}\ast(e^{\prime
})^{-1}\ast\alpha_{v_{1}}^{-1}$. Since $\left.  \overline{g}\widetilde{f}%
\right\vert _{Y^{\left(  1\right)  }}$ is level preserving, $e$ and
$e^{\prime}$ project to the same interval in $%
%TCIMACRO{\U{211d} }%
%BeginExpansion
\mathbb{R}
%EndExpansion
$, so we have two key facts:%
\begin{equation}
\operatorname*{height}\left(  \beta_{e}\right)  \leq\operatorname*{height}%
\left(  e\right)  +2q_{i} \label{line: bounded height}%
\end{equation}

and%
\begin{equation}
\beta_{e}\subseteq\widetilde{M}_{i}. \label{line: narrowness of loops}%
\end{equation}
We will extend $H_{t}^{\left(  0\right)  }$ over all of $Y^{\left(  1\right)
}$ by filling in each of the $\beta_{e}$ with disks. To make $H_{t}$ proper,
we arrange that finitely many such disks intersect any given $\widetilde{K}%
_{i}^{\left[  -i,i\right]  }$. The argument is essentially the same as the one
used to construct $\overline{g}$.

If $e$ lies outside $\widetilde{K}_{i+1}$, then $\beta_{e}$ also lies in
$Y-\widetilde{K}_{i+1}$; so it can be filled in missing $\widetilde{K}%
_{i}^{\left[  -i,i\right]  }$.

For the edges $\left\{  e_{\rho}\right\}  $ lying in $\widetilde{K}_{i+1}$,
fact (\ref{line: narrowness of loops}) ensures that the loops $\left\{
\beta_{e_{\rho}}\right\}  $ also lie in $\widetilde{K}_{i+1}$. Note that there
is a uniform bound on the heights of the $\left\{  e_{\rho}\right\}  $; this
is by $J$-equivariance, since each is a lift of one of the finitely many edges
in $K_{i}$. So fact (\ref{line: bounded height}) ensures that there is an
upper bound on the heights of the $\left\{  \beta_{e_{\rho}}\right\}  $. By
applying Lemma \ref{Lemma 5}, we can fill in all but finitely many with disks
missing $\widetilde{K}_{i}^{\left[  -i,i\right]  }$.\medskip
\end{proof}

\section{General Conclusions\label{Section: General Conclusions}}

We conclude by assembling our main theorems in their most general forms. In
contrast to Theorem \ref{main} from the introduction, there are no
restrictions on the number of ends of $Y$. In all cases, $Y$ is a simply
connected, locally compact ANR admitting a $%
%TCIMACRO{\U{2124} }%
%BeginExpansion
\mathbb{Z}
%EndExpansion
$-action by covering transformations generated by a homeomorphism
$j:Y\rightarrow Y$. Conclusions involve the topology at infinity of $Y$, in
particular, proper homotopy invariants in dimensions $<2$. The conclusions
vary, depending on assumptions placed on $j$. All serious work has been
completed. Here we need only combine the proper 1- and 2-equivalences obtained
in Propositions \ref{Prop: proper 1-equivalence downstairs},
\ref{Prop: proper 2-equivalence downstairs}, and
\ref{Prop: proper 2-equiv with only coaxial hypothesis} and Corollary
\ref{Cor: lifting proper 1- and 2-equivalences} with the analyses of the model
spaces in Propositions \ref{Prop: Understanding the universal cover of X},
\ref{Proposition: End properties of the model space}, and
\ref{Prop: Model Z-space} and Remark
\ref{Remark. Subcases of infinite-ended Y}.

In the first theorem, no additional requirements are placed on $j$. The
conclusions involve the number of ends of $Y$ and the action of $j$ on those
ends. Most, if not all, were previously known; nevertheless, the theorem
illustrates the effectiveness of our approach and places subsequent theorems
in the context of some familiar and useful facts.

\begin{theorem}
\label{Theorem: General Theorem 1}Let $Y$ be a simply connected, locally
compact ANR admitting a $%
%TCIMACRO{\U{2124} }%
%BeginExpansion
\mathbb{Z}
%EndExpansion
$-action by covering transformations generated by a homeomorphism
$j:Y\rightarrow Y$ and let $J=\left\langle j\right\rangle $. Then $Y$ is
$J$-equivariantly properly 1-equivalent to its universal model space
$\widetilde{X}_{\Gamma}$. As a result, $Y$ has 1,2, or infinitely many ends. Moreover,

\begin{enumerate}
\item \label{Assertion 1 of General Theorem 1}if $Y$ is 2-ended, then $j$
fixes the ends of $Y$, the action is cocompact, and $Y$ is equivariantly
proper $1$-equivalent to a line;

\item if $Y$ is infinite-ended, then precisely one or two ends are stabilized
by $j$, with the rest occurring in $J$-transitive families, each member of
which has a neighborhood in $Y$ that projects homeomorphically onto a
neighborhood of an end of $J\backslash Y$;

\item $Y$ has uncountably many ends if and only if $J\backslash Y$ has
uncountably many $\pi_{1}$-null ends (as defined in \S \ref{topology}).
\end{enumerate}
\end{theorem}

\begin{corollary}
If an infinite-ended finitely presented group $G$ acts properly and
cocompactly on a simply connected, locally compact ANR $Y$, and $g\in G$ has
infinite order, then $\left\langle g\right\rangle \backslash Y$ has
uncountably many ends.
\end{corollary}

\begin{remark}
Although a simple connectivity hypothesis on $Y$ was built into our
constructions, in anticipation of the most interesting theorems, it was not
needed to obtain a proper 1-equivalence between $\left\langle j\right\rangle
\backslash Y$ and $X_{\Gamma}$. Hence, the conclusions of Theorem
\ref{Theorem: General Theorem 1} are valid provided $Y$ is connected.
\end{remark}

For the next theorem and its corollary, we add the assumption that $j$ is coaxial.

\begin{theorem}
\label{Theorem: General Theorem 2}Let $Y$ be a simply connected, locally
compact ANR admitting a $%
%TCIMACRO{\U{2124} }%
%BeginExpansion
\mathbb{Z}
%EndExpansion
$-action by covering transformations generated by a coaxial homeomorphism
$j:Y\rightarrow Y$ and let $J=\left\langle j\right\rangle $. Then $Y$ is
properly 2-equivalent to its model $%
%TCIMACRO{\U{2124} }%
%BeginExpansion
\mathbb{Z}
%EndExpansion
$-space $\widetilde{X}_{\Gamma}$ via maps that are $J$-equivariant on
1-skeleta. As a result, $Y$ has 1,2,or infinitely many ends, and:

\begin{enumerate}
\item \label{Assertion 1 of General Theorem 2}if $Y$ is 2-ended, the
$J$-action is cocompact and $Y$ is properly $2$-equivalent to a line,

\item if $Y$ is $1$-ended, then $Y$ is properly $2$-equivalent to $\Lambda
^{+}\times%
%TCIMACRO{\U{211d} }%
%BeginExpansion
\mathbb{R}
%EndExpansion
$, where $\Lambda^{+}$ is an infinite rooted tree, and the $%
%TCIMACRO{\U{2124} }%
%BeginExpansion
\mathbb{Z}
%EndExpansion
$-action on $\Lambda^{+}\times%
%TCIMACRO{\U{211d} }%
%BeginExpansion
\mathbb{R}
%EndExpansion
$ is generated by a homeomorphism $\sigma_{\infty}\times t$, where
$\sigma_{\infty}$ fixes the root of $\Lambda^{+}$ and $t\left(  r\right)
=r+1$,

\item if $Y$ is infinite-ended, then $J$ stabilizes exactly one or two of
those ends of $Y$; and

\begin{enumerate}
\item if two ends are stabilized, $Y$ is properly 2-equivalent to $%
%TCIMACRO{\U{211d} }%
%BeginExpansion
\mathbb{R}
%EndExpansion
\cup\left(  \sqcup_{i\in%
%TCIMACRO{\U{2124} }%
%BeginExpansion
\mathbb{Z}
%EndExpansion
}\Omega_{i}\right)  $, where $\left\{  \Omega_{i}\right\}  _{i\in%
%TCIMACRO{\U{2124} }%
%BeginExpansion
\mathbb{Z}
%EndExpansion
}$ is a collection of isomorphic rooted trees with the root of $\Omega_{i}$
identified to $i\in%
%TCIMACRO{\U{211d} }%
%BeginExpansion
\mathbb{R}
%EndExpansion
$, and the $J$-action on $%
%TCIMACRO{\U{211d} }%
%BeginExpansion
\mathbb{R}
%EndExpansion
\cup\left(  \sqcup_{i\in%
%TCIMACRO{\U{2124} }%
%BeginExpansion
\mathbb{Z}
%EndExpansion
}\Omega_{i}\right)  $ is an extension of translation by $+1$ on $%
%TCIMACRO{\U{211d} }%
%BeginExpansion
\mathbb{R}
%EndExpansion
$,

\item if only one end is stabilized, then $Y$ is properly $2$-equivalent to
$\left(  \Lambda^{+}\times%
%TCIMACRO{\U{211d} }%
%BeginExpansion
\mathbb{R}
%EndExpansion
\right)  \cup\left(  \cup\Omega_{m,n}\right)  $, where $\left\{  \Omega
_{m,n}\right\}  $ is a locally finite collection of rooted trees, with each
$\Omega_{m,n}$ attached at its root to a vertex of $\Lambda^{+}\times\{n\}$,
and for each fixed $m$, $\left\{  \Omega_{m,n}\right\}  _{n\in%
%TCIMACRO{\U{2124} }%
%BeginExpansion
\mathbb{Z}
%EndExpansion
}$ is a pairwise disjoint subcollection on which $J$ acts transitively, taking
roots to roots,
\end{enumerate}

\item $Y$ has uncountably many ends if and only if $J\backslash Y$ has
uncountably many null ends.
\end{enumerate}

\noindent Furthermore, if $j$ is strongly coaxial, the proper 2-equivalences
can be chosen to be $%
%TCIMACRO{\U{2124} }%
%BeginExpansion
\mathbb{Z}
%EndExpansion
$-equivariant.
\end{theorem}

\begin{corollary}
\label{Corollary: to General Theorem 2}Let $Y$ be a simply connected strongly
locally finite CW complex admitting a $%
%TCIMACRO{\U{2124} }%
%BeginExpansion
\mathbb{Z}
%EndExpansion
$-action by covering transformations generated by a coaxial homeomorphism
$j:Y\rightarrow Y$. Then $Y$ is 1-, 2-, or infinite-ended. Moreover,

\begin{enumerate}
\item if $Y$ is 2-ended, both ends are simply connected and the $%
%TCIMACRO{\U{2124} }%
%BeginExpansion
\mathbb{Z}
%EndExpansion
$-action fixes those ends;

\item \label{Assertion 2 of Corollary to General Theorem 2}if $Y$ is 1-ended,
that end is semistable and $\operatorname*{pro}$-$\pi_{1}\left(  Y,r\right)  $
can be represented by an inverse sequence of surjections between finitely
generated free groups%
\[
F_{1}\twoheadleftarrow F_{2}\twoheadleftarrow F_{3}\twoheadleftarrow\cdots
\]
and $\operatorname*{pro}$-$H_{1}\left(  Y;%
%TCIMACRO{\U{2124} }%
%BeginExpansion
\mathbb{Z}
%EndExpansion
\right)  $ can be represented by an inverse sequence of surjections between
finitely generated free abelian groups%
\[%
%TCIMACRO{\U{2124} }%
%BeginExpansion
\mathbb{Z}
%EndExpansion
^{n_{1}}\twoheadleftarrow%
%TCIMACRO{\U{2124} }%
%BeginExpansion
\mathbb{Z}
%EndExpansion
^{n_{2}}\twoheadleftarrow%
%TCIMACRO{\U{2124} }%
%BeginExpansion
\mathbb{Z}
%EndExpansion
^{n_{3}}\twoheadleftarrow\cdots
\]

\item if $Y$ is infinite-ended, the $%
%TCIMACRO{\U{2124} }%
%BeginExpansion
\mathbb{Z}
%EndExpansion
$-action fixes precisely one or two ends with the others having trivial
stabilizers. All non-fixed ends are simply-connected. If two ends are fixed,
those ends are simply connected as well. If just one end is fixed, that end is
semistable with $\operatorname*{pro}$-$\pi_{1}\left(  Y,r\right)  $
representable by an inverse sequence like the one described in Assertion
(\ref{Assertion 2 of Corollary to General Theorem 2}). Similarly,
$\operatorname*{pro}$-$H_{1}\left(  Y;%
%TCIMACRO{\U{2124} }%
%BeginExpansion
\mathbb{Z}
%EndExpansion
\right)  $ is representable by a sequence like the one found in Assertion
(\ref{Assertion 2 of Corollary to General Theorem 2}), with all nontrivial
contributions coming from the fixed end.
\end{enumerate}
\end{corollary}

\begin{remark}
\emph{I}f desired, the $%
%TCIMACRO{\U{2124} }%
%BeginExpansion
\mathbb{Z}
%EndExpansion
$-equivariance of the proper 2-equivalences on 1-skeleta can be used to
specify the action of $J$ on $\operatorname*{pro}$-$\pi_{1}\left(  Y,r\right)
$ and $\operatorname*{pro}$-$H_{1}\left(  Y;%
%TCIMACRO{\U{2124} }%
%BeginExpansion
\mathbb{Z}
%EndExpansion
\right)  $. In particular, they will look like the easily understood $%
%TCIMACRO{\U{2124} }%
%BeginExpansion
\mathbb{Z}
%EndExpansion
$-actions on $\operatorname*{pro}$-$\pi_{1}\left(  \Lambda^{+}\times%
%TCIMACRO{\U{211d} }%
%BeginExpansion
\mathbb{R}
%EndExpansion
\right)  $ and $\operatorname*{pro}$-$H_{1}\left(  \Lambda^{+}\times%
%TCIMACRO{\U{211d} }%
%BeginExpansion
\mathbb{R}
%EndExpansion
;%
%TCIMACRO{\U{2124} }%
%BeginExpansion
\mathbb{Z}
%EndExpansion
\right)  $ generated by $\sigma_{\infty}\times t$, where $\sigma_{\infty}$
fixes the root of $\Lambda^{+}$ and $t\left(  r\right)  =r+1.$
\end{remark}

\bibliographystyle{amsalpha}
\bibliography{bibliography}
{}

\end{document}